\documentclass[12pt]{amsart}
\usepackage{graphicx,epsfig}
\usepackage{bbding,amssymb,amsfonts,amsmath,amsthm,dsfont,wasysym,pifont,stmaryrd}
\usepackage{epstopdf,yfonts}
\usepackage{hyperref}

\usepackage{mathrsfs}
\usepackage{enumerate}

\usepackage{url}
\usepackage{tikz}

\newtheorem{theorem}{Theorem}[section]
\newtheorem*{theorem'}{Theorem }

\newtheorem{corollary}[theorem]{Corollary}

\newtheorem*{convention}{Convention}

\newtheorem{prop}[theorem]{Proposition}
\newtheorem{proposition}[theorem]{Proposition}

\newtheorem{cor}[theorem]{Corollary}
\newtheorem{remark}[theorem]{Remark}
\newtheorem{example}[theorem]{Example}
\newtheorem{lemma}[theorem]{Lemma}
\newtheorem{definition}[theorem]{Definition}

\newcommand{\1}{\mathds{1}}

\renewcommand{\)}{\right)}

\renewcommand{\^}[1]{\hat{#1}}
\renewcommand{\~}[1]{\overline{#1}}

\renewcommand{\geq}{\geqslant}
\renewcommand{\leq}{\leqslant}

\renewcommand{\>}{\right\rangle}
\newcommand{\8}{\infty}

\renewcommand{\a}{\alpha}
\newcommand{\Aut}{\text{Aut}}

\renewcommand{\b}{\beta}

\renewcommand{\Cap}[1]{\underset{#1}{\cap }}
\newcommand{\ch}[1]{\check{#1}}

\newcommand{\eye}{\sphericalangle}

\newcommand{\f}{\varphi}

\newcommand{\frakH}{\mathfrak{H}}
\newcommand{\g}{\gamma}
\newcommand{\G}{\Gamma}

\newcommand{\I}{\mathcal{I}}

\renewcommand{\int}{\varint}
\newcommand{\Isom}{\mathrm{Isom}}

\newcommand{\Lim}[1]{\underset{#1}{\lim}}

\newcommand{\N}{\mathbb{N}}

\newcommand{\Ps}{{\mathbb{P}}}

\newcommand{\EE}{\mathbb E}

\newcommand{\R}{\mathbb{R}}

\newcommand{\SL}{\textnormal{SL}}
\newcommand{\SO}{\textnormal{SO}}

\newcommand{\Sum}[1]{\underset{#1}{{\sum} }}

\newcommand{\Z}{\mathbb{Z}}

\DeclareMathOperator{\Prob}{Prob}

\DeclareMathOperator{\Haar}{Haar}

\newcommand{\bd}{\partial_\eye}
\newcommand{\bdh}{\partial_\infty^{\mathrm{horo}}}
\newcommand{\bdr}{\partial_r}

\newcommand{\eps}{\varepsilon}

\title{Random Walks and Boundaries of CAT(0) Cubical complexes}
\author{Talia Fern\'os, Jean L\'ecureux, Fr\'ed\'eric Math\'eus}
\thanks{T. Fern\'os was partially supported by NSF Grant number DMS-1312928, and UNCG New Faculty Summer Excellence Research Grant.\\
J. L\'ecureux was partially supported by Projet ANR-14-CE25-0004 GAMME}

\begin{document}

\begin{abstract}
We show under weak hypotheses that the pushforward $\{Z_no\}$ of a random-walk to a CAT(0) cube complex converges to a point on the boundary. We introduce the notion of squeezing points, which allows us to consider the convergence in either the Roller boundary or the visual boundary, with the appropriate hypotheses. This study allows us to show that any nonelementary action necessarily contains regular elements, that is, elements that act as rank-1 hyperbolic isometries in each irreducible factor of the essential core. 
\end{abstract}
\maketitle

\tableofcontents

\section{Introduction}

Let $\mu$ be a probability measure on a group $\G$. Pick elements $g_i$ independently and at random according to the law $\mu$. The random walk on $\G$ is defined as the sequence $Z_n=g_1g_2\dots g_n$. An important aspect of the study is to understand the asymptotic behavior of the random walk $Z_n$. 

A typical way of understanding how elements of a given group behave is to make the group act on a metric space $X$.   Fixing a base point $o\in X$, one can then study the sequence of points $\{Z_no\}$. If the space $X$ is sufficiently nice, one can hope for the convergence of this sequence of points in some geometric compactification of $X$. The first example of this is due to Furstenberg, where the space in question is the hyperbolic plane \cite{Furstenberg1963}. A powerful motivation for this kind of result is Oseledec' Theorem for random walks on subgroups of $\SL_n(\R)$ \cite{Oseledec}; it can be interpreted as a form of convergence of the random walk to a point in the (visual) boundary of the symmetric space $\SL_n(\R)/\SO_n(\R)$ \cite{KaimanovichOseledts}. 

These types of questions have been studied by many authors. Let us give a few results in this direction. 
   The typical setting in which these  results will hold is in the presence of negative curvature, or at least spaces with hyperbolic-like properties. The fundamental paper of Kaimanovich \cite{Kaimanovich} proves this convergence for hyperbolic groups, and in many situations when $X$ has some kind of negative curvature. Let us also mention the work of Kaimanovich and Masur, treating the case of the mapping class group of a surface acting on its Teichm\"uller space \cite{KaimanovichMasur}, and the work of Gautero and Math\'eus on groups acting on $\R$-trees \cite{GauteroMatheus}. More recently, a nice result of Maher and Tiozzo  \cite{MaherTiozzo} proves the convergence to the boundary for groups acting on  (not necessarily proper) hyperbolic spaces.  In the CAT(0) setting, there are also some partial results. Ballman  treats the case of groups acting on non-positively curved  rank-one manifolds \cite{Ballmann89}. For general CAT(0) spaces, Karlsson and Margulis \cite{KarlssonMargulis} prove  convergence to the visual boundary, but they assume that the random walk goes to infinitiy at positive speed, which can be difficult to check in general.\\

In this paper, we are interested in the case when $X$ is a CAT(0) cube complex. These complexes attracted a lot of attention recently as they play an essential role in  Agol's proof of the virtual Haken conjecture for 3-manifolds (an outstanding problem in the theory of 3-manifolds which relied essentially on the work of Wise) \cite{Agol}, \cite{WiseMalnormal}. There are many examples of CAT(0) cube complexes and groups acting on them. Apart from the fundamental groups of hyperbolic 3-manifolds, one can think of right-angled Artin groups, Coxeter groups, and small cancellation groups, among many others. Let us also emphasize that there are interesting examples of CAT(0) cube complexes which are not proper. For example, the Higman group $\langle a_i,i\in \Z/n\Z\mid a_ia_{i+1}a_i^{-1}=a_{i+1}^2 \rangle$ (with $n\geq 4$) acts (non-properly) on a CAT(0) square complex \cite{Martin}. Another example is given by diagram groups \cite{Farley} (the complex in this case may fail to be finite dimensional).

CAT(0) cube complexes admit two natural metrics which in turn give rise to the visual boundary and the Roller boundary. The boundary which will be the most relevant for our study is the Roller boundary (see \S\ref{SecRoller}) though we will also consider visual boundary (see \S\ref{RW and Visual}). 

\begin{theorem}
Let $X$ be a finite-dimensional CAT(0) cube complex with an essential and nonelementary action of the group $\G$. Furthermore, assume  that $\G$ stabilizes each irreducible factor of $X$. Then, for any admissible measure $\mu \in \Prob(\G)$ and every $o\in X$, almost surely one has that $Z_n o$ converges to some point in the Roller boundary. 
\end{theorem}

A question which is related to the convergence to the boundary is the speed at which the random walk goes to infinity, called the \emph{drift}. This drift is defined as the limit $\lambda=\lim\frac{d(Z_n o,o)}{n}$ (see \S\ref{drift}). The random walk on a non-amenable group $\G$, endowed with some word metric, always has positive drift \cite{Guivarc'h1980}. For general actions however the positivity is not clear at all. In some cases, establishing the positivity of the drift helps to prove  convergence to the boundary, as in \cite{KarlssonMargulis}. In our case however, we deduce the positivity of the drift from the convergence, and prove the following (see Theorem \ref{theorem pos drift}):

\begin{theorem}
Let $X$ be a finite-dimensional CAT(0) cube complex with a non-elementary, essential action of the group $\G$. Assume $\mu\in \Prob(\G)$ is admissible and has finite first moment. Then almost surely we have $\lim\limits_{n\to \infty} \frac{d(Z_n o,o)}{n}>0$.\\
\end{theorem}

While the Roller boundary is the most useful for us, a CAT(0) cube complex is also a CAT(0) space, and therefore is endowed with another natural boundary: its visual boundary $\bd X$. From a measurable point of view, in many cases, there should be an isomorphism between the two boundaries. Indeed, under suitable assumptions, the Roller boundary as well as the visual boundary is the Furstenberg-Poisson boundary of $(\G,\mu)$ by \cite{Fernos} and by \cite{KarlssonMargulis}. 
However, there is in general no natural map which is everywhere defined between the two boundaries.  It is nevertheless possible to define some partial maps: for example, to a point $\eta$ in the Roller boundary, one can associate the set of possible limit points in the visual boundary of any sequence converging to $\eta$ (see Section \ref{Roller vs Visual} for more details).  It might happen that, for certain points of the Roller boundary, this set is reduced to a point. It turns out that we are able to prove that, for almost every limit point of the random walk, this is the case (see Section \ref{RW on visual} and  Proposition \ref{asspecial}).  

After proving the convergence of the random walk to the Roller boundary, it is natural to wonder what happens with the visual boundary. As mentionned above, Karlsson and Margulis proved convergence of the random walk to the visual boundary for groups acting on CAT(0) spaces under the assumption of finite first moment and positivity of the drift \cite{KarlssonMargulis}. By considering actions on CAT(0) cube complexes, we are able to remove the moment condition and prove the following: 

\begin{theorem}
Let $X$ be a finite-dimensional, irreducible, CAT(0) cube complex with a non-elementary, essential action of the group $\G$. Assume that $\G$ stabilizes each irreducible factor of $X$. Then, for any admissible $\mu \in \Prob(\G)$ and for every $o\in X$, almost surely the sequence $(Z_n o)$ converges to some point in the visual boundary.\\
\end{theorem}

Once we have proved the convergence to the boundary, we can better understand   the dynamics of the random walk $Z_n$. Say that a geodesic $\ell$ in $X$ is \emph{contracting} if the projection on $\ell$ of any ball disjoint from $\ell$ has uniformly bounded diameter. An isometry of $X$ is called \emph{contracting} if it is a hyperbolic isometry with a contracting axis. The fundamental paper of Caprace and Sageev \cite{CapraceSageev} proves that for irreducible complexes, any non-elementary action has contracting elements. We are able to prove that these elements occur with high probability in the random walk:

\begin{theorem}
Let $X$ be a finite-dimensional irreducible CAT(0) cube complex with an essential and non-elementary action  of the group $\G$. Then for any admissible $\mu\in \Prob(\G)$, we have that almost surely 
$$\lim_{n\to +\infty}\frac{1}{n} \vert \{k \leq n\mid Z_k \textrm{ is contracting }  \}\vert=1$$
\end{theorem}

As an application, we can generalize one of the main results of \cite{CapraceSageev} in the case of reducible complexes, where there cannot be any contracting isometries. The best that one can hope is for  elements which act as contracting isometries in each irreducible factor (of the essential core). These are called \emph{regular elements}. Caprace and Sageev prove that such elements do exist, under the additional assumption that $\G$ is a lattice in $\Aut(X)$ \cite[Theorem D]{CapraceSageev} (see also \cite{CapraceZadnik} for the case of general CAT(0) spaces). 
Using the theorem above, we can get rid of this assumption:

\begin{theorem}
Let $X$ be a finite-dimensional  CAT(0) cube complex with an essential and non-elementary action of the group $\G$. Then there exists regular elements in $\G$.
\end{theorem}

In fact, not only do regular elements exist, but they will occur in the random walk with high probability (see Corollary \ref{RegularEverywhere!}).
The existence of such elements has some strong consequences about the asymptotic properties of the $\G$-orbits in $X$ \cite{Link}. \\

Our strategy of proof for all these theorems is inspired by some classical results: Kaimanovich \cite{Kaimanovich} for the convergence to the boundary and Guivarc'h and Raugi \cite{GR85} for the positivity of the drift. However, to be able to apply these strategies, we are forced to understand the dynamics on the boundary. An important tool for us is the notion of \emph{regular points} of the boundary (see \S\ref{sec:reg}). These special points were introduced in the paper \cite{Fernos} and exhibit strong contracting properties very useful to us.

 Another distinctive feature of our proof is that, in opposition for example to \cite{Kaimanovich}, we use the identification of the Furstenberg-Poisson boundary (proved in \cite{Fernos}) in order to prove the convergence to the boundary. More precisely, we use that there is a boundary map from the Furstenberg-Poisson boundary of $\G$ to the Roller boundary of $X$, and that the essential image of this map is contained in the set of regular points. Then the contracting properties of the regular points are sufficient to ensure the convergence.

\subsection*{Acknowledgments}

The authors would like to thank Uri Bader, Ruth Charney, Indira Chatterji, Amos Nevo, L'Institut Henri Poincar\'e, and the first and third named authors would like to thank the Laboratoire de Math\'ematique d'Orsay.

\section{Generalities about Random Walks}

\subsection{Generalities and Notation}

Let us start with setting up some notation. In what follows, $\G$ is a discrete countable group. 
We fix an \emph{admissible} probability measure $\mu\in\Prob(\G)$, meaning that the semigroup generated by the support of $\mu$ is $\G$.

We define the random walk on $\G$ as follows. Let $\Omega=\G^{\N}$ and $\Ps$ be the probability measure on $\Omega$ defined by $\Ps=\delta_e\times \mu^{\N^*}$. The space $\Omega$ is the \emph{space of increments}. If $\omega\in \Omega$, we denote by $g_i(\omega)$ the $i$th element of the sequence $\omega$. As is customary in probability theory, we often omit the $\omega$ and write only  $g_i$.

Our main object of interest is the random walk on $\G$, which is the sequence of random variables $Z_n:\Omega\to \G$ defined by  $Z_n(\omega)=g_1(\omega)g_2(\omega)\dots g_n(\omega)$, or for short $Z_n=g_1\dots g_n$.

\subsection{The Furstenberg-Poisson Boundary}

The proof of our results will use an important tool: the Furstenberg-Poisson boundary of $(\G,\mu)$. This boundary is a  space designed to encode the asymptotic properties of the sequences $(Z_n)$. We will briefly recall the definition of this space and the key results that we need. The interested reader might consult \cite{Furman}, \cite{Kaimanovich}, \cite{BaderShalom}, \cite{BFICM}, or \cite{Fernos} for more information.

One possible definition is as follows. We denote by $S:\Omega\to \Omega$ the ``shift" map defined by $S(\omega_0,\omega_1,\dots,\omega_n,\dots)=(\omega_0\omega_1,\omega_2,\dots,\omega_n,\dots)$. 

\begin{definition}
The Furstenberg-Poisson boundary is the space $B$ of ergodic components of the action of $S$ on $(\Omega,\Haar\otimes \mu^{\N^*})$. It is equipped with the pushforward $\nu$ of the measure $\Ps$ by the projection $\Omega\to B$.
\end{definition}

So the Furstenberg-Poisson boundary is a measure space equipped with an action of $\G$ and a probability measure $\nu$ whose class is preserved by $\G$. 

We will need to understand the Poisson boundary of finite index subgroups. More precisely, we need the following result, which is proved in \cite[Lemma 4.2]{Furstenberg}.

\begin{lemma}\label{Poissonfi}
Let $(B,\nu)$ be the Furstenberg-Poisson boundary of $(\G,\mu)$, and let $\G_0<\G$ be a subgroup of finite index. Then there exists an admissible measure $\mu_0\in\Prob(\G_0)$ such that the Furstenberg-Poisson boundary of $(\G_0,\mu_0)$ is $\G_0$-equivariantly isomorphic to $(B,\nu)$.
\end{lemma}

The Furstenberg-Poisson boundary presents very strong ergodic properties. This was first observed in \cite{KaimanovichErgodic} and more recently generalized and used in \cite{BFICM}. In the following, we denote by $(B_-,\nu_-)$ the Furstenberg-Poisson boundary of $(\G,\check\mu)$, where $\check\mu\in\Prob(\G)$ is defined by $\check\mu(g)=\mu(g^{-1})$.

\begin{theorem}\label{IsomErgodic}
Let $Y$ be a separable metric space endowed with an action of $\G$ by isometries. Then:
\begin{itemize}
\item Any $\G$-equivariant measurable map $B\to Y$ is essentially constant;
\item Any $\G$-equivariant measurable map $B_-\times B\to Y$ is essentially constant.
\end{itemize}
\end{theorem}

\subsection{Stationary Measures}

Let $\G$ act continuously on some topological space $K$.  A measure $\lambda\in \Prob(K)$ is \emph{stationary} if $\mu *\lambda=\lambda$, in other words if $\int_ \G g_*\lambda \,d\mu(g)=\lambda$. It is a general fact that if $\G$ acts continuously on some compact space $K$ then there always exists some stationary measure on $K$ \cite[Lemma 1.2]{Furstenberg1963}.

We will use the following important consequence of the Martingale Convergence Theorem \cite[Lemma 1.3]{Furstenberg1963}:

\begin{theorem}\label{martingale}
Let $\lambda$ be a stationary measure on the compact space $K$. For $\Ps$-almost every $\omega\in\Omega$ there exists $\lambda_\omega\in\Prob(K)$ such that $Z_n(\omega)\lambda$ converges to $\lambda_\omega$. Furthermore we have $\lambda=\int_\Omega \lambda_\omega\, d\mathbb P(\omega)$.
\end{theorem}

It is easy to check that the measure $\nu$ on $B$ is always $\mu$-stationary. Furthermore, if $\lambda$ is a $\mu-$stationary measure on a compact space $K$, then by Theorem \ref{martingale} we get a map $\Omega\to \Prob(K)$ given by $\omega\mapsto \lim Z_n(\omega)\lambda$. This map is clearly $S$-invariant, so it factors through a map $B\to \Prob(K)$.

The above theorem can be generalized to Polish spaces. Let $Y$ be a Polish space, with a continuous action of $\G$. We endow $\Prob(Y)$ with the topology of weak-* convergence, when seen as a dual of the space of bounded continuous function. It is again a Polish space with a continuous action of $\G$. The following is proved in \cite[Lemma 3.2]{BenoistQuintStationnaire}:

\begin{theorem}\label{MartingalePolish}
Let $Y$ be a Polish space with a continuous $\G$-action.
Assume that $\lambda$ is a stationary probability measure on $Y$. Then  for $\Ps$-almost every $\omega\in\Omega$ there exists $\lambda_\omega\in\Prob(Y)$ such that $Z_n(\omega)\lambda$ converges to $\lambda_\omega$. Furthermore we have $\lambda=\int_\Omega \lambda_\omega\, d\mathbb P(\omega)$. 
\end{theorem}

\begin{corollary}\label{cor:uniquestationmeas}
Let $\G$ act continuously on some Polish space $Y$. Assume that there is a unique $\G$-equivariant map $\f: B\to \Prob(Y)$. Then there is a unique stationary measure on $Y$.
\end{corollary}

\begin{proof}
Let $b\mapsto \lambda_b$ be a $\G$-equivariant map. Then it is easy to check that $\lambda=\int \lambda_b \,d\nu(b)$ is a stationary measure on $Y$.

Now let us turn to the uniqueness. Let $\lambda$ be a stationary measure on $Y$. We know from the Martingale Convergence Theorem that $Z_n\lambda$ converges to some measure $\lambda_b$, and $b\mapsto\lambda_b$ is a $G$-equivariant map from $B $ to $\Prob(Y)$. Hence we have $\lambda_b=\f(b)$. 

Since we have also $\lambda=\int_\Omega \lambda_b \,d\nu(b)$, we see that $\lambda$ is uniquely defined.

\end{proof}

\section{General Facts about CAT(0) Cube Complexes}

In this section we collect some general results about  CAT(0) cube complexes. We assume some familiarity with these basic concepts. We refer the interested reader to \cite{CapraceSageev}, \cite{NevoSageev} or \cite{Fernos} for more information.

\begin{convention}
 In what follows all the complexes we consider will be finite-dimensional and second countable.
\end{convention}

\begin{remark}
The restriction to second countable complexes is needed for ergodic-theoretic arguments, but is not essential to our purpose. Indeed, if a countable group $\Gamma$ acts on a complex $X$, then it is easy to check that there is a sub-complex $Y\subset X$ which is second countable and $\Gamma$-invariant.
\end{remark}

\subsection{Sageev-Roller Duality and the Roller Boundary}\label{SecRoller}

Let $X$ be a finite-dimensional CAT(0) cube complex. In what follows, we identify $X$ with  its set of vertices. We endow $X$ with the combinatorial distance (also called the $\ell^1$-distance): the distance between any  two vertices is defined as their distance in the 1-skeleton of $X$.

We denote by $\frakH$ the collection of half-spaces of $X$. If $h\in\frakH$, we denote by $h^*$ the half-space which is the complement of $h$. For $h,k\in\frakH$, we say that $h$ is transverse to $k$ and write $h\pitchfork k$ if the four intersections $h\cap k$, $h\cap k^*$, $h^*\cap  k$ and $h^*\cap k^*$ are nonempty.

Fix $v\in X$ and consider the collection $U_v= \{h\in\frakH: v\in h\}$.
The \emph{Sageev-Roller Duality} is then obtained via the following observation: 

$$\Cap{h \in U_v}{} h = \{v\}.$$

This shows that every vertex $v$ is uniquely defined by the set $U_v$.
This immediately yields an embedding $X\hookrightarrow 2^\frakH$ obtained by $v\mapsto U_v$. 
Thanks to this duality, it may at times be simpler to confuse $v$ and $U_v$, though we will make an effort to make the distinction. 
 The metric on $X$ becomes then
$d(x,y) = \frac{1}{2}\#(U_x\triangle U_y).$

In the following definition, we identify $X$ with its image in $2^\frakH$. 

\begin{definition} 
The \emph{Roller Compactification} is denoted by $\~X$ and is the closure of $X$ in $2^\mathfrak H$. The \emph{Roller Boundary} is then $\partial X = \~X \setminus X$. 
\end{definition}

Let $\eta\in\~X$. Then, $\eta$ is the limit of some sequence $(x_n)$ of vertices of $X$, and by definition, $U_\eta$ is the pointwise limit of $U_{x_n}$. We say that $\eta$ is in the  half-space $h$ if $h\in U_\eta$. In this way we have a partition $\~X=h\sqcup h^*$.

It is possible (and more common in the literature) to define the Roller boundary as a subset of $2^\frakH$ satisfying some combinatorial conditions (\emph{totality} and \emph{consistency}). This  turns out to be equivalent to the construction described above. 

In the Roller boundary, the vertices of $X$ correspond to  $U\in  \~X\subset 2^\frakH$ satisfying the \emph{descending chain condition}: any decreasing sequence of elements of $U$ is eventually constant.

On the opposite side, we find \emph{nonterminating elements}. These special elements were defined by Nevo and Sageev \cite{NevoSageev} as follows:

\begin{definition}
An element $v\in\~X$ is \emph{nonterminating} if every finite descending chain can be extended, i.e. given any $h\in U_v$ there exists $k\in U_v$ such that
$k\subset h$.

The set of nonterminating elements is denoted by $\partial_{NT}X$.
\end{definition}

\subsection{Medians and Intervals}\label{secmedian}

The \emph{interval} between two points $x$ and $y$ in $X$ is defined as $\I(x,y)=\{z\mid d(x,z)+d(z,y)=d(x,y)\}$.

It is easy to  see that 
$$\I(x,y)=\{z\in X\mid U_x\cap U_y\subset U_z\}$$

This definition extends easily to  the Roller boundary: the interval between $v,w\in\~X$ is defined as $\I(v,w)=\{m\in \~X\mid U_v\cap U_w\subset U_m\}$. This interval structure endows $X$ with the structure of a \emph{median  space} \cite{Chatterji_Niblo,Nica},  which can be extended to the Roller compactification as follows.

The \emph{median} of three points $u,v,w\in \~X$ is the point $m=m(u,v,w)$ defined by the formula
$$\label{eq:median}
U_m = (U_u\cap  U_v) \cup ( U_v\cap  U_w)\cup ( U_w\cap  U_u),
$$

Equivalently, the point $m$ is the unique point 
$$\{m\}= \I(u,v) \cap \I(v,w) \cap \I(w,u).$$

While CAT(0) cube complexes can be quite wild, the structure of intervals is somewhat tamable by the following (see
\cite[Theorem 1.16]{BCGNW}).

\begin{lemma}\label{embedintervals} Let $v,w\in\~{X}$.  
Then the vertex interval $\I(v,w)$ isometrically embeds into $\~{\Z^D}$ 
(with the standard cubulation) where $D$ is the dimension of $X$.
\end{lemma}

\subsection{Product Structure}

A CAT(0) cube complex is said to be \emph{reducible} if it can be expressed as a nontrivial product. Otherwise, it is said to be \emph{irreducible}. A CAT(0) cube complex $X$ with half-spaces $\frakH$, admits a product decomposition $X=X_1\times \cdots \times X_n$ if and only if there is a decomposition 
$$\frakH= \frakH_1 \sqcup \cdots \sqcup \frakH_n$$
such that if $i\neq j$ then $h_i\pitchfork h_j$ for every $(h_i, h_j)\in \frakH_i\times \frakH_j$ and $X_i$ is the CAT(0) cube complex on half-spaces $\frakH_i$. 

Furthermore, we have the following \cite[Proposition 2.6]
{CapraceSageev}:

\begin{proposition}
The decomposition 
$$X=X_1\times \cdots \times X_n$$
where each $X_i$ is irreducible, is unique (up to permutation of the factors).
The group $\Aut(X)$ contains $\Aut(X_1)\times \dots \times \Aut(X_n)$ as a finite index subgroup.
\end{proposition}

 Therefore, if $\G$ acts on $X$ by automorphisms, then there is a subgroup of finite index which preserves the product decomposition.

We also note that the Roller compactification behaves quite well with respect to products: indeed, if $X=X_1\times \dots\times X_n$, then we have $\~X=\~X_1\times \dots\times \~X_n$.

\section{Actions on CAT(0) Cube Complexes}

We denote by $\bd X$ the visual  boundary of $X$.

\begin{definition}
 An isometric action on a CAT(0) space is said to be elementary if there is a finite orbit in either the space or the visual boundary. 
\end{definition}

Caprace and Sageev developed a theory of non-elementary actions on a CAT(0) cube complex. They first prove that there is a nonempty ``essential core" where  the action is well behaved. Let us now develop the necessary terminology and recall the key facts.

\begin{definition}
Let $\G<\Aut(X)$. A half-space $h\in \frakH$ is called \emph{shallow} if for some (hence all) $x\in X$, the set $\G x\cap h$ is at bounded distance from $h^*$.

The action of $\G$ on $X$ is \emph{essential} if no half-space is shallow.
\end{definition}

As mentioned above, it is always possible to reduce a non-elementary action to an essential action \cite[Proposition 3.5]{CapraceSageev}:

\begin{proposition}
Let $\G$ be a group with a non-elementary action on $X$. There exists a non-empty subcomplex $Y\subset X$ which is $\G$-invariant and on which the $\G$-action is essential and nonelementary.
\end{proposition}

Suppose that $\G$ is acting on $X$ a CAT(0) cube complex. A simple but powerful concept introduced by Caprace and Sageev is that of flipping a half-space. A half-space $h\in \frakH$ is said to be $\G$-flippable if there is a $g\in \G$ such that $h^* \subset g h$. The following is due to Caprace and Sageev:

\begin{lemma}[Flipping Lemma]\label{flip}
 Let $\G$ act non-elementarily on the CAT(0) cube complex $X$. If $h\in \frakH$ is essential, then $h$ is $\G$-flippable. 
\end{lemma}

Another very important operation on half-spaces studied by Caprace and Sageev is the notion of double skewering. The following is  again from \cite{CapraceSageev}:

\begin{lemma}[Double Skewering Lemma]\label{Double Skewering Lemma}
 Let $\G$ act non-elementarily on the CAT(0) cube complex $X$. If $h\subsetneq k$ are two essential half spaces, then there exists an $g \in \G$ such that $$g k\subsetneq h \subsetneq k.$$
\end{lemma}

 For the proof of the following lemma we refer to \cite[Lemma 2.28]{CFI}.

\begin{lemma}
Let $\G\to \Aut(X)$ be a non-elementary and essential action. Let $\G_0<\G$ be the finite index subgroup which preserves every factor. Then the action of $\G_0$ on each irreducible factor of $X$ is again non-elementary and essential.
\end{lemma}

\section{Separation Properties of Hyperplanes and the Regular Boundary}

\subsection{Strongly Separated Hyperplanes}

The following notion was introduced by Behrstock and Charney \cite{Behrstock_Charney}, in their study of Right Angled Artin Groups. Caprace and Sageev later used this to find a powerful criterion for irreducibility of CAT(0) cube complexes. 

Recall that two half-spaces $h$ and $k$  are \emph{transverse} if $\^ h \cap \^k \neq \varnothing$. This is equivalent to the four intersections  $h\cap k$, $h\cap k^*$, $h^*\cap  k$ and $h^*\cap k^*$ being nonempty.

\begin{definition}
Two walls $\^h$ and $\^k$ are called \emph{strongly separated} if there is no wall which is transverse to both $\^h$ and $\^k$. Two half-spaces are said to be \emph{strongly separated} if their walls are so. 
\end{definition}

Clearly if a complex is not irreducible, then it is can not contain strongly separated pairs. This turns out to be both necessary and sufficient:

\begin{theorem}[\cite{CapraceSageev}]
Let $X$ be a CAT(0) cube complex such that the action of $\Aut(X)$ is essential and nonelementary. There exists a pair of strongly separated half-spaces if and only if $X$ is irreducible. 
\end{theorem}

\subsection{The Combinatorial Bridge}
\label{bridge}

Behrstock and Charney showed that the CAT(0) bridge connecting two strongly separated walls is a finite geodesic segment \cite{Behrstock_Charney}. In \cite{CFI} this idea is translated to the ``combinatorial", i.e. median setting for general walls. For our purposes, it suffices to consider strongly separated pairs. Most of what follows is from or adapted from \cite{CFI} and \cite{Fernos}.

Let $h_1\subset h_2$ be a nested pair of halfspaces. Consider the set of pairs of points in $h_1\times h_2^*$ minimizing the distance between $h_1$ and $h_2^*$, that is
$$M_{h_1,h_2} =\{(x, y) \in h_1\times h_2^* :  \text{ if }(a,b) \in h_1\times h_2^* \text{ then } d(x,y)\leq d(a,b)\}.$$

Observe that we immediately have that $(x,y) \in M_{h_1,h_2}$ then $x,y \in X$. The following lemma is taken from \cite[Section 2.G]{CFI}

\begin{lemma}\label{bridgeunique}
If $h_1\subset h_2$ are strongly separated nested half-spaces, then there exists a unique pair of vertices $(p_1,p_2)$ such that $M_{h_1,h_2} = \{(p_1,p_2)\}$.
\end{lemma}

\begin{definition}
For $h_1\subset h_2$ the \emph{combinatorial bridge} connecting $h_1$ and $h_2^*$ is the union of intervals between  minimal distance pairs: 
$$B(h_1,h_2^*)= \bigcup_{(x,y)\in M_{h_1,h_2}}{}\mathcal{I}(x,y).$$ 
\end{definition}

Lemma \ref{bridgeunique} rewrites as follows:

\begin{lemma}
Let $h_1\subset h_2$ be strongly separated nested halfspaces. Then there exists $p_1\in h_1$ and $p_2\in h_2$ such that $B(h_1,h_2)=\I(p_1,p_2)$.
\end{lemma}

If $h_1$ and $h_2$ are strongly separated, define the \emph{length} of the bridge $b(h_1,h_2)$ as the distance from $p_1$ to $p_2$. We also call this length the \emph{distance} between the two strongly separated half-spaces $h_1$ and $h_2$.

%
%
%
%

%
%

\begin{definition}
Two hyperplanes $\hat h$ and $\hat k$ are \emph{super strongly separated} (or \"uber-separated in \cite{CFI}) if for any hyperplanes $\hat h'$ and $\hat k'$ intersecting respectively $\hat h$ and $\hat k$, we have $\hat h'\cap \hat k'=\emptyset$.
\end{definition}

Two half-spaces are super strongly separated if their walls are so. Note that if $h\subset k\subset l$ are pairwise strongly separated, then $h$ and $l$ are super strongly separated. So if $X$ is irreducible with a non-elementary and essential automorphism group, there always exists a pair of super strongly separated half-spaces.

Super strong separation has the following consequence on the bridge. If $A$ is a subset of $X$ and $r>0$, we denote by $V_r(A)$ the $r$-neighborhood of $A$ (always in the combinatorial distance).

\begin{proposition}\label{BridgeGeod}
Let $h\subset k$ be a pair of super strongly separated half-spaces, and $\ell$ be the length of the bridge. If $x\in h$ and $y\in k^*$ then $I(x,y)\subset V_\ell(b(h,k))$.
\end{proposition}

\begin{proof}
See \cite[Lemma 3.5]{CFI}
\end{proof}

\subsection{The Regular Boundary}\label{sec:reg}

Using strongly separated hyperplanes, it is possible to define a notion of a \emph{regular} boundary. This notion was first defined in \cite{Fernos} (and independently in \cite{KarSageev}, where it was called ``strongly separated points").

\begin{definition}
Assume $X$ is irreducible. A point $\xi\in\partial X$ is called \emph{regular} if for every $h_1,h_2\in U_\xi$ there is $k\in U_\xi$ such that $k\subset h_1\cap h_2$ and $k$ is strongly separated both from $h_1$ and $h_2$.
The set of regular points of $X$ is denoted by $\partial_r X$.
\end{definition}

This notion has a natural extension to products:

\begin{definition}
Let $X=X_1\times \cdots\times X_n$ be the decomposition of $X$ into irreducible factors. The set of \emph{regular points} of $X$ is defined as
$$\partial_r X=\partial_r X_1\times\cdots\times \partial_r X_n$$

The \emph{regular boundary} of $X$ is the closure of $\partial_r X$ in $\overline X$. We denote  the regular boundary by $R(X)$.

\end{definition}

\subsection{On Descending Chains of Half-Spaces}

In the irreducible case, regular points can be characterized as follows. Recall that a \emph{descending chain} is a sequence $(h_n)_{n\in\N}$ of half-spaces such that $h_{n+1}\subsetneq h_n$. Vertices in $X$ are characterized as the set of points $x\in \~X$ satisfying the \emph{descending chain condition}: there is no (infinite) descending chain in $U_x$.

\begin{proposition}{\cite[Proposition 7.4]{Fernos}}\label{Rank1Char}
Let $X$ be an irreducible complex, and $\alpha\in\~X$. The follwing are equivalent:
\begin{enumerate}[(i)]
\item $\alpha\in\bdr X$
\item There exists an infinite descending chain $(h_n)_{n\in\N}$ of pairwise strongly separated half-spaces such that $\alpha\in h_n$.
\end{enumerate}
\end{proposition}

It is possible to analyze more precisely the descending chains containing $\alpha$. We first record  the following.

\begin{lemma}\label{lemma Annoying cases}
 Let $\{h_n\}\in \frakH$ be an infinite descending chain of half-spaces. If $k\in \frakH$ such that $k\cap h_n \neq \varnothing$ for all $n$ then one of the following is true:
\begin{enumerate}
\item[(a)] There is an $N$ such that $k\pitchfork h_n$ for all $n>N$.
\item[(b)] There is an $N$ such that $k\supset h_n$ for all $n>N$.
\end{enumerate}
In particular if the sequence $\{h_n\}$ is composed of pairwise strongly separated half-spaces then Case \emph{(b)} holds. 
\end{lemma}

\begin{proof}
Fix $n$. Our assumption that $k\cap h_n \neq \varnothing$ implies that one of the following cases hold:

\begin{enumerate}
\item $h_n^* \subset k$
\item $h_n \supset k$
\item $h_n \pitchfork k$
\item $h_n \subset k$
\end{enumerate}

Now, observe that since there are finitely many half-spaces in-between any two, and hence the collection of all $n$ which satisfy conditions (1) and (2) is finite. Next observe that if there is an infinite subsequence  which satisfies property (3) (respectively property (4)) then $h_n$ satisfies  property (3) (respectively property (4)) for all $n$ sufficiently large. 

Of course, if the sequence $\{h_n\}$ is pairwise strongly separated, it follows that condition (3) can hold for at most one $n$.
\end{proof}

We can now prove the following.

\begin{lemma}\label{singleton}
Let $(s_n)$ be an infinite descending chain of pairwise strongly separated half-spaces. Then $\bigcap\limits_{n\in\N} s_n$ is a singleton.

If $X$ is an irreducible complex and $\alpha\in\bdr X$, then any infinite descending chain $(h_n)_{n\in\N}$ of half-spaces containing $\alpha$ satisfies that $\bigcap\limits_{n\in\N} h_n=\{\alpha\}$.
\end{lemma}

\begin{proof}
The fact that $\bigcap\limits_{n\in\N} s_n$ is a singleton is proved in \cite[Corollary 7.5]{Fernos}.

Now consider an arbitrary descending chain $(h_n)$ containing $\alpha$. By the first part of the lemma, it is sufficient to prove that for every $m\in\N$ there exists $n\in\N$ such that $h_n\subset s_m$. Since $h_m$ and $s_n$ both contain $\alpha$, we have $h_m\cap s_n\neq \emptyset$ for every $m,n$. 

Fix $m$. By Lemma \ref{lemma Annoying cases}, we know that either for every $n$ large enough we have $h_n\subset s_m$ (in which case we are done), or for every $n$ large enough $h_n\pitchfork s_m$. In the second case, apply now Lemma \ref{lemma Annoying cases} to $k=s_{m+1}$. By strong separation, we know that $h_n$ is not transverse to $s_{m+1}$ for $n$ large. So we must have $h_n\subset s_{m+1}$, which contradicts the fact that $h_n\pitchfork s_m$.

Hence we have proved that for every $m$ and every $n$ large enough we have $h_n\subset s_m$. So $\bigcap_{n\in\N} h_n\subset\bigcap_{n\in\N} s_n=\{\alpha\}$, and by assumption $\alpha$ is in $\bigcap_{n\in\N} h_n$, which proves that we have equality.
\end{proof}

The previous lemmas deal with one boundary point. For two points, we have the following:

\begin{proposition}\label{biinfinitechain}
Let $X$ be an irreducible complex and $\alpha,\beta\in\bdr X$. Assume that $\alpha \neq \beta$. Then there exists a sequence $(s_n)_{n\in\Z}$ of pairwise strongly separated half-spaces, with $s_{n+1}\subset s_n$, and such that 
$s_n\in U_\alpha\setminus U_\beta$ for all $n$.
\end{proposition}

\begin{proof}

 Proposition \ref{Rank1Char} guarantees that we can find two sequences, each of pairwise strongly separated half-spaces $\{s_n(\a): n\in \N\}\subset U_\a$ and $\{s_n(\b): n\in \N\} \subset U_\b$.
 
Since $\alpha\neq \beta$, there exists $h\in U_\alpha\setminus U_\beta$ (and hence $h^*\in U_\beta$). By Lemma \ref{lemma Annoying cases}, there exists an $N$ such that for every $n>N$ we have $s_n(\alpha)\subset h$ and $s_n(\beta)\subset h^*$. Discarding finitely many half-spaces, we may and shall assume that these two equalities hold for every $n$.
 We define $s_n = s_n(\a)$ for $n\geq 0$ and $s_n = s_{-n}(\b)^*$ for $n<0$. Then almost all the conditions on the chain $(s_n)$ are clear. The only thing remaining to check is the strong separation of $s_0$ and $s_{-1}$. But a half-space $k$ which is transverse to both $s_0$ and $s_{-1}$ must be transverse to $s_0(\beta)$ which  is in-between, contradicting the strong separation of $s_0(\beta)$ and $s_1(\beta)$. 

\end{proof}








\begin{lemma}\label{medianinX}
Let $X$ be an irreducible complex and $\alpha,\beta,\g$ be pairwise distinct points of $\~X$ with  $\alpha$ and $\beta$ regular. Then the median point $m(\alpha,\beta,\g)$ is a vertex in $X$.
\end{lemma}

\begin{proof}
Consider $m= m(\a,\b, \g)$. We claim that $m\in X$ and to this end we show that $U_m$ satisfies the descending chain condition.  Recall that $ U_m= (U_\a\cap U_\b)\cup(U_\b\cap U_\g)\cup(U_\g\cap U_\a) \subset U_\a\cup U_\b$. Assume by contradiction that $U_m$ contains an infinite descending chain. Then, up to discarding finitely many (and possibly relabeling $\a$ and $\b$), we may assume by Lemma \ref{singleton} that the chain belongs to $U_\a$ and hence $m=\a$. This means that $U_\a \subset U_\b\cup U_\g$. By Proposition \ref{biinfinitechain} there is an infinite descending chain of pairwise strongly separated half-spaces in $U_\a\setminus U_\b\subset U_\g$. Once more by Lemma \ref{singleton} we deduce that $\a= \g$, a contradiction.

\end{proof}

\begin{lemma}\label{IinterX}
Let $X$ be an irreducible complex and $\alpha\in\bdr X$ and $\beta\in\overline X$ with $\beta\neq\alpha$. Then $\I(\alpha,\beta)\cap X\neq \varnothing$.
\end{lemma}

\begin{proof}
It suffices to show that the set $U_{\alpha}\cap U_\beta$ satisfies the descending chain condition (see for example \cite[Lemma 2.3]{NevoSageev}). Assume that there exists a decreasing sequence of half-spaces $(h_n)$ with $h_n\in U_\alpha\cap U_\beta$. Then by Lemma \ref{singleton} the intersection of all the half-spaces $h_n$ is reduced to $\{\a\}$. Since we also have $\beta\in h_n$ for all $n$, this implies $\alpha=\beta$, contradicting the assumption.

\end{proof}



\section{Comparing Various Boundaries}

So far, we have introduced two  boundaries of CAT(0) cube complexes: the Roller boundary and the regular boundary. There are also other interesting constructions. In this section, we aim to compare these.

\subsection{The Roller and Visual Boundaries}\label{Roller vs Visual}

Let us start by the most common boundaries of CAT(0) cubical complexes: the Roller boundary  $\partial X$ and the visual boundary $\bd X$.

The following theorem, which is due to P.E. Caprace and A. Lytchak \cite[Theorem 1.1]{CapraceLytchak}, is very useful in this situation.

\begin{theorem}\label{thm:CapLytch}
Let $(X_i)_{i\in I}$ be a filtering family of closed convex subsets of a finite-dimensional CAT(0) space $X$. Then either the intersection $\bigcap\limits_{i\in I} X_i$ is not empty, or the intersection $\bigcap\limits_{i\in I}\bd X_i$ of their boundaries is not empty, and has intrinsic radius less than $\pi/2$.
\end{theorem}

The intrinsic radius less than $\pi/2$ gives the existence of a ``canonical" center. 

For the purpose of the following, we shall consider a half-space as the closure of the CAT(0) convex hull of the vertices contained in the half-space. Consider a point $\a$ in the Roller boundary $\partial X$ and its collection of half-spaces $U_\a$. This is a filtering family of closed convex spaces, so we can apply Theorem \ref{thm:CapLytch}. Since $\a$ contains an infinite descending chain, the intersection of all half-spaces in $U_\a$ with  $X$  is empty. So we get:

\begin{corollary}\label{cor:Q(a)}
Let $\a\in\partial X$. Let $Q(\a)=\bigcap\limits_{h\in U_\a} \bd h$. Then $Q(\a)$ is not empty. 

Furthermore, the map associating to $\a$ the center of $Q(\a)$ is an $\Aut(X)$-equivariant map from $\partial X$ to $\bd X$.
\end{corollary}

In general, there is more than one point in $Q(\a)$, and it might also happen that $Q(\a)=Q(\b)$ for $\a\neq \b$. For example, take $\a=(\8,0), \b=(\8, 1)\in \~\Z^2$, then $Q(\a)=Q(\b)$ corresponds to the geodesic of slope 0.

Now let us attempt to find some kind of inverse map. 
Let $\xi\in \partial_\sphericalangle X$,
let $g:[0,\infty)\to X$ be a geodesic asymptotic to $\xi$. We say that a half-space $h\in\frakH$ is \emph{transverse} to $\xi$ if for every $R>0$ there exists $t_R\geq0$ such that the $R$-neighborhood
of the image of the geodesic ray $g|_{(t_R,\infty)}$ is contained in $h$.  We denote by $T_\xi$ the set of half-spaces transverse to $\xi$.
This set does not depend on the particular choice of the geodesic $g$ in the class of $\xi$.

\begin{lemma}\label{lemTxi}
 Let $X$ be a CAT(0) cube complex
 and let $\xi\in\partial_\sphericalangle X$.  
Then the set $T_\xi$ is not empty and  $\Cap{h\in T_\xi}h\neq \varnothing$.

Furthermore, $T_\xi$ contains an infinite descending chain.
\end{lemma}

\begin{proof}
See \cite[Lemma 2.27]{CFI}, where it is proved that $T_\xi$ is not empty, contains an infinite descending chain, and that it satisfies the partial choice and consistency condition (hence has a non-empty intersection  in $\overline X$).
\end{proof}

We denote the intersection by $\~X_\xi=\Cap{h\in T_\xi} h$. It is a subset of $\overline X$ (and by Lemma \ref{lemTxi} is disjoint from $X$). We will also denote by $X_\xi$ the subset of $\a\in\~X_\xi$ such that $U_\a\setminus T_\xi$ satisfies the descending chain condition (which is trivially satisfied if $U_\a\setminus T_\xi=\varnothing$).

We have defined two maps: the map $\a\mapsto Q(\a)$ from the Roller boundary to (closed subsets of) the visual boundary, and the map $\xi\mapsto \~X_\xi$ from the visual boundary to (closed subsets of) the Roller boundary. 
These two maps are somehow inverse to one another.

\begin{lemma}\label{ainQxi}
Let $\a\in \partial X$. Let $Q(\a)$ be as in Corollary \ref{cor:Q(a)}, and let $\xi\in Q(\a)$. Then $\a\in \~X_\xi$.

Conversely, let $\xi\in\bd X$ and $\a\in \~X_\xi$. Then $\xi\in Q(\a)$.
\end{lemma}

\begin{proof}
Let us prove the first part: let $\a\in\partial X$ and $\xi\in Q(\a)$. Let $h\in T_\xi$. Assume that $h\not\in U_\a$, which means that $h^*\in U_\a$. Since $\xi\in Q(\a)$, this implies that $\xi$ is in the visual boundary of $h^*$. So there is a geodesic ray $g_0$ converging to $\xi$ which is contained in $h^*$. Any other geodesic ray converging to $\xi$ will be at bounded distance from $g_0$. This implies that $h\not\in T_\xi$, which is a contradiction. So we have $h\in U_\a$. It follows that $\a$ is contained in the intersection of all half-spaces in $U_\a$, which is $\~X_\xi$. 

Now let $\xi\in\bd X$ and $\a\in \~X_\xi$. Let $h\in U_\a$, and let us prove that $\xi\in\bd h$. If $h\in T_\xi$, then the result is clear. Since $\a\in \~X_\xi$, we cannot have $h^*\in T_\xi$. Now assume that neither  $h$ or $h^*$ are in $T_\xi$. Pick a geodesic asymptotic to $\xi$. If this geodesic is in $h$, then we are done. If not, since $h^*\not\in T_\xi$, we see that this geodesic stays at bounded distance from $h$. This means that $\xi\in\bd h$  (in fact even $\xi\in\bd\hat h$). This proves that every $\xi$ is in the boundary of every half-space in $U_\a$. So $\xi\in Q(\a)$.

\end{proof}

We also record the following.

\begin{lemma}\label{lem:geodsector}
Let $\a\in\partial X$ and $\xi\in Q(\a)$. Let $o\in X$. Then the CAT(0)-geodesic ray from $o$ to $\xi$ is contained in the interval $I(o,\a)$.
\end{lemma}

\begin{proof}
Let $\overline I(o,\a)$ be the closure of $I(o,\a)\cap X$ in $X\cup\bd X$. Since $I(o,\a)=\bigcap\limits_{h\in U_o\cap U_\a} h$, we have $\overline I(o,\a)=\bigcap\limits_{h\in U_o\cap U_\a} ((h\cap X)\cup \bd h)$. 
So  $\xi\in Q(\a)$ implies that $\xi\in\overline I(o,\a)$.  As $I(o,\a)\cap X$ is a convex subset of $X$ (for the $\ell^1$ metric and hence also for the $\ell^2$ metric), it follows that the geodesic from $o$ to $\xi$ is contained in $I(o,\a)\cap X$.
\end{proof}

\subsection{Squeezing Points}\label{Squeezing}

The notion of a squeezing point will be indispensable in Section \ref{RW and Visual} where we connect the behavior of the random walk with the visual boundary. We begin by establishing the notion for points in the Roller boundary, and then discuss the notion for points in the visual boundary. 

\begin{definition}
Assume that $X$ is irreducible. We say that a point $\eta\in\partial X$ is \emph{squeezing} if there exists an $x\in X$ and an $r>0$ such that there exist infinitely many pairs of super strongly separated $h\subset k$ at distance $r$, with $\eta\in h\cap k$ and $x\in h^*\cap k^*$.

If $X$ is not irreducible, a squeezing point is one that is squeezing in each factor.
\end{definition}

\begin{remark}
For an irreducible complex $X$ a point $\eta \in \partial X$ is \emph{contracting} if there is a bi-infinite decreasing sequence of pairwise strongly separated half-spaces in $U_\eta$ which are at consecutive distance  $r$.
The reader may then note the similarity between a \emph{squeezing point} and  a \emph{contracting point}.  Contracting points are necessarily squeezing, but the converse does not hold in general. Both squeezing and contracting points are necessarily regular. 
\end{remark}

Recall the definition of $Q(\eta)$ from  Corollary \ref{cor:Q(a)}. The properties of squeezing points are summarized in the following lemma.

\begin{lemma}\label{lem:specialvisual}
Let $\eta\in \partial X$ be a squeezing point. Then there exists $\xi\in\bd X$ such that $Q(\eta)=\{\xi\}$. Furthermore, any sequence of vertices $(x_n)$ converging to $\eta$ in the Roller boundary also converges to $\xi$ in the visual boundary.
\end{lemma}

\begin{proof}
 Let $x\in X$ and  $r>0$ be such that there exists an infinite sequence of super strongly separated half-spaces $h_i\subset k_i$ at distance $r$, with $\eta\in h_i\cap k_i$ and $x\in h_i^*\cap k_i^*$.  
 
Let us prove first that $Q(\eta)$ is a singleton. Assume that there exist $\xi,\xi'\in Q(\eta)$. Let $g$ and $g'$ be the geodesic rays from $x$ to $\xi$ and $\xi'$ respectively. Then for every $i$, both the rays $g$ and $g'$ cross both hyperplanes $\hat h_i$ and $\hat k_i$. By Lemma \ref{BridgeGeod}, they have to be in the $r$-neighborhood of the bridge $b(h_i,k_i)$. Furthermore, the bridge $b(h_i,k_i)$ crosses exactly the $r$ hyperplanes separating $h_i$ from $k_i$. So its diameter (for the combinatorial distance $d$) is at most $r$. Hence its diameter for the distance $d'$ is at most $C $, for some $C>0$ (depending only on $r$). It follows that the two geodesic rays $g$ and $g'$ are at distance $C'$ for some (fixed) $C'>0$ when they travel in $h_i\cap k_i^*$. 

Since $h_i$ and $k_i$ can be arbitrarily far from $x$, it follows that $g$ and $g'$ are at distance $C'$ from each other at arbitrarily large distance from $x$. By convexity of the distance in a CAT(0) space, it follows that they are always  at distance at most $C'$ from each other. Hence $\xi=\xi'$.

Now let $(x_n)$ be a sequence of vertices of $X$ converging to $\eta$. Let $g_n$ be the geodesic ray from $x$ to $x_n$. We have to prove that $g_n$ converges to $g$ uniformly on every compact set. Let $R>0$ and let $(h_i,k_i)$ be half-spaces in the sequence defined above which are at distance $>R$ from $x$. For $n$ large enough, we see that $x_n$ belongs to $h_i\cap k_i$, so that $g_n$ crosses $\hat h_i$ and $\hat k_i$. So using the same argument as above, for every $R>0$ and every $t<R$, we have $d'(g_n(t),g(t))<C'$.

To avoid cumbersome notation for the remainder of the proof only we shall denote both the CAT(0) metric on $X$ and on Euclidean space by $d$.  Fix $\eps>0$ small.  Consider the comparison triangle $\bar x$, $\bar\g_n(R)$, and $\bar \g(R)$ in the Euclidean plane $\R^2$. Let $t<R\eps/C'$, $p=\g_n(t)$, and $q=\g(t)$, and consider again the  points $\bar p$ and $\bar q$ in $\R^2$ on the segments $[\bar x \bar \g_n(R)]$ and $[\bar x \bar g(R)]$ respectively and both at distance $t$ from $\bar x $. Since we know that $d(\bar\g_n(R),\bar\g(R))=d(\g_n(R),\g(R))\leq C'$,  using the Law of Similar Triangles  we see that $d(\bar p,\bar q)\leq \frac{t C'}{R}<\eps$. By definition of CAT(0) spaces, it follows that $d(p,q)<\eps$. In other words, we have, for all $t<R\eps/C'$, $d(\g_n(t),\g(t))<\eps$. The result follows.

\end{proof}

Lemma \ref{lem:specialvisual} justifies the following:

\begin{definition}
 Assume that $X$ is irreducible. An element $\xi \in \partial_\eye X$ is said to be \emph{squeezing} if for some (and hence all) $x\in X$ there is an $r>0$ and infinitely many pairs of super strongly separated $h\subset k$ at distance $r$, with $x\in h^*\cap k^*$ such that geodesic ray from $x$ asympotic to $\xi$ crosses the walls $\^h$ and $\^k$.
 \end{definition}

Recall from Corollary \ref{cor:Q(a)} that there is an $\Aut(X)$-equivariant map $\partial X \to \partial_\eye X$. This together with Lemma \ref{lem:specialvisual} yields:

\begin{lemma}\label{lem: visaula and roller squeezing bijection}
 There is an $\Aut(X)$-equivariant bijection between the squeezing points in $\partial X$ and the squeezing points in $\partial_\eye X$. 
\end{lemma}

This justifies the following definition:

\begin{definition}
 The interval between two visual squeezing points $\xi_-, \xi_+\in \partial_\eye X$ is defined as $\I(\xi_-, \xi_+) := \I(Q^{-1}(\xi_-),Q^{-1}( \xi_+))$ which is a subset of the Roller compactification $\~X$. 
\end{definition}

\subsection{A Quotient of the Roller Boundary.}

The set of boundary points has a natural partition into cubical subcomplexes, which is especially interesting for points that are not nonterminating). The following definition is due to Guralnik \cite{Guralnik}.

\begin{definition}
Let $\alpha,\beta\in\partial X$. We say that $\alpha$ is equivalent to $\beta$, denoted by $\alpha\sim\beta$, if the symmetric difference between $U_\alpha$ and $U_\beta$  is finite. The equivalence class of $\alpha$ is denoted $[\alpha]$.
\end{definition}

\begin{definition}
The \emph{extended metric} on $\~X$ is the function $d:\~X\times\~X\to \R\cup \{+\infty\}$ defined by the same formula as on $X$:
$$d(\a,\b)=\frac12 \#(U_\a\triangle U_\b) $$
\end{definition}

The extended distance between two points $\a$ and $\b$ is finite if and only if we have $\a\in[\b]$. For every $\a$, this endows $[\a]$ with a distance. In fact, $[\a]$ is a CAT(0) cubical complex in its own right, the half-spaces of $[\a]$ being the half-spaces of $X$ which separate two points in $[\a]$.

\begin{lemma}\label{existschain}
For every $\a\in \partial X$, there exists $\xi\in \bd X$ such that $[\a]\subset \~X_\xi$. 

Furthermore, there exists a descending chain $(h_n)_{n\in \N}$ of half-spaces such that $[\alpha]\subset\bigcap_{n\in\N}h_n$.
\end{lemma}

\begin{proof}
Let $Q=Q(\a)$ be defined as in Corollary \ref{cor:Q(a)}, and fix $\xi\in Q$. Then $\a\in \~X_\xi$ by Lemma \ref{ainQxi}. It follows that $[\a]\subset \~X_\xi$. Finally, $T_\xi$ contains an infinite descending chain by Lemma \ref{lemTxi}.
\end{proof}

\subsection{Subcomplexes as Decreasing Intersections}

We defined in the previous section an extended distance $d:\~X\times\~X\to \R\cup\{+\infty\}$, which partitions $\~X$ into cubical subcomplexes. We aim to write these subcomplexes as intersections of half-spaces in $\~X$.

\begin{lemma}\label{Z=Ybar}
Let $(h_n)_{n\geq 1}$ be a descending chain of half-spaces, and $Z=\bigcap\limits_{n\geq 1} h_n$. Then $Z$ is the Roller compactification of some subcomplex $Y\subset \~X$.
\end{lemma}

\begin{proof}
Indeed, consider the set of half-spaces $\frakH'\subset\frakH$ such that $h\cap Z$ and $h^*\cap Z$ are both nonempty. Then by \cite[Lemma 2.6]{CFI} (see also \cite[Proposition 2.10]{Fernos}) there is an isometric embedding of the CAT(0) cube complex associated to $\frakH'$ into $\~X$, whose closure is exactly $Z$. 

We note that  $\frakH'$ is given by all half-spaces which are transverse to infinitely many $h_n$.
\end{proof}

\begin{lemma}\label{dim<}
Let $Y\subset \overline X$ be a subcomplex disjoint from $X$. Then $\dim(Y)<\dim(X)$.
\end{lemma}

\begin{proof}
  Let $\mathcal{D}$ be a maximal collection of pairwise transverse half-spaces in $X$. Let us denote by $H_Y$ the set of half-spaces containing $Y$. We aim to prove that there is a $k\in \mathcal{D}$ such that $k$ or $k^*\in H_Y$ and so $\^k$ does not participate in any maximal cube of $Y$. 
  
  We begin by observing that if $h,k\in \frakH$ are such that  $k\cap h\neq \varnothing$ and $k^*\cap h\neq \varnothing$ then one of the following hold:
  
  \begin{enumerate}
\item $k\subset h$;
\item $k^* \subset h$;
\item $h\pitchfork k$.
\end{enumerate}
  
Consider $h_{n+1}\subsetneq h_n $  an infinite  descending chain in $H_Y$, which exists by Lemma \ref{existschain} (since $Y=[y]$ for any $y\in Y$). We now show that $\mathcal{D}\cap (H_Y\sqcup H_Y^*)\neq\varnothing$.
  
By contradiction, assume this is not the case, i.e. that if $k\in \mathcal{D}$ then $k\cap Y$ and $k^*\cap Y$ are both nonempty, and in particular, $k\cap h_n, k^*\cap h_n\neq \varnothing$ for each $n$. Therefore, for each $k\in \mathcal D$ and $n\in \N$, we are in one of the situations (1)--(3) above. Since in between any two half-spaces there are finitely many, and  $\mathcal D$ is finite, there must be an $N$ such that if $n>N$ then $h_n\pitchfork k$ for every $k \in \mathcal D$. It follows that for every $k\in \mathcal D$ and for every $n$ large enough we have $h_n\pitchfork k$. This of course contradicts the maximality of $\mathcal D$.  


This shows that any maximal family of pairwise transverse half-spaces must have non-trivial intersection with $H_Y\sqcup H_Y^*$ and hence the dimension of $Y$ is less than $D$.
\end{proof}

\begin{lemma}
Let $\xi_0\in\partial X$. There exists $k\leq \dim (X)$ and a family $(h^1_m)_{m\geq 0}$, $(h_m^2)_{m\geq 0},\dots, (h_m^k)_{m\geq 0}$ of descending chain of  half-spaces  such that 
$$\overline{[\xi_0]} =\bigcap_{i=1}^k\bigcap_{m\geq 0}h_m^i$$
\end{lemma}

\begin{proof}\label{lemma finitely many chains}
We argue by induction on the dimension. If $\dim(X)=1$, then the result is clear. 

Assume the lemma holds for every complex of dimension $<\dim(X)$. By Lemma \ref{existschain}, there exists a descending chain $(h_m)$  whose intersection contains $[\xi_0]$ (and since half-spaces are closed, it also contains  $\overline{[\xi_0]}$). Let $Z=\bigcap h_m$.
By Lemma \ref{Z=Ybar},  $Z$ is isomorphic to the Roller compactification of some complex $Y$.

By Lemma \ref{dim<} we have $\dim(Y)<\dim(X)$. We also know that  $\overline{[\xi_0]}\subset \overline Y$. If $\xi_0\in Y$ then $\overline{[\xi_0]}=Z$ and there is nothing left to prove. If not, then by induction there exists finitely many chains of  half-spaces in $Y$ such that $\overline{[\xi_0]}$ is the intersection of all these half-spaces. These half-spaces lift to half-spaces of $X$. 
To conclude the proof we observe that the lift of these half-spaces in $X$ form again a descending chain (indeed, any non-empty intersection of walls or half-spaces in $X$  projects to a non-empty intersection in $Y$).
\end{proof}

\subsection{Horofunction Boundary}

\label{sec:horo}

Let $(X,d)$ be a metric space. Let us recall the construction of the horoboundary of $X$. Fix an origin $o\in X$. For $x\in X$, consider the function $h_x:X\to \R$ defined by $h_x(y)=d(y,x)-d(o,x)$. This defines an embedding $\iota$ from $X$ to the set $\mathcal C(X)$ of continuous function on $X$.

\begin{definition}
The \emph{horocompactification} $\overline{X}^h$ is the closure of $\iota(X)$ in $\mathcal C(X)$. The \emph{horoboundary} of $X$ is $\bdh(X)=\overline{X}^h\setminus X$.

A function in $\bdh(X)$ (and sometimes even in $\overline{X}^h$) is called a \emph{horofunction}.
\end{definition}

Because every function $h_x$ is actually 1-Lipschitz and satisfies $h_x(o)=0$, it follows from the Arzela-Ascoli theorem that the horocompactification is indeed a compact space (regardless of the topology of $X$). Furthermore, the horoboundary, as a topological space, does not depend on the choice of the origin $o$ (a different choice would just translate the horofunctions by a constant).

It is well-known that for a proper CAT(0) space $X$ (with its CAT(0) metric),  the horoboundary is the same as the visual boundary, denoted $\bd X$. 

\begin{remark}
This notion of horoboundary is not the usual one because we consider the topology of  convergence on \emph{compact} subsets, and not on \emph{bounded} ones. For proper spaces, the two notions are of course equivalent. The main advantage of our definition is that it produces a compact space. However, there are two possible inconveniences: the first one is that the space is no longer open in its compactification, and the second one is that for general spaces this construction might produced more points than desired. To avoid the confusion, these limit points are called \emph{metric functionals} instead of horofunctions in  \cite{KarlssonGouezel}. However, when considering the horoboundary with the $\ell^1$ as we do above,  there are no additional points and so we stick to the more standard terminology.
\end{remark}

Now let us go back to our situation when $X$ is a CAT(0) cube complex. Recall from \S\ref{SecRoller} that the distance on $X$ can be calculated as $d(x,y)= \frac{1}{2}\#(U_x\triangle U_y).$ 

The following is an unpublished result of Bader and Guralnik, and seems to be well-known to experts. We include a proof for completeness.

\begin{proposition}\label{horo=Roller}
The horocompactification (respectively the horoboundary) of the set of vertices of $X$ is equivariantly homeomorphic to the Roller compactification (respectively the Roller boundary) of $X$.

Furthermore, for every $\xi\in\overline X$, if $m$ is the median point of $\xi,x$ and $o$, then the horofunction associated to $\xi$ is defined by
$$h_\xi(x)=d(m,x)-d(m,o).$$
\end{proposition}

Let us start with a lemma which is of independent interest. Recall (from \S\ref{secmedian}) that the median point of $x,y,z$ is the unique point contained in the intersection $I(x,y)\cap I(y,z)\cap I(z,x)$. 

\begin{lemma}\label{mediancontinuous}
The map $m:\overline{X}\times\overline{X}\times\overline{X}\to\overline{X}$ which associates to a triple of points their median is continuous.
\end{lemma}

\begin{proof}
Let $x,y,z\in\overline X$, and $m=m(x,y,z)$. The definition of the median translates easily to get that
$$U_m=(U_x\cap U_y)\cup (U_y\cap U_z)\cup (U_z\cap U_x).$$
It is straightforward to verify that this is in fact defines a continuous map  $2^\frakH \times 2^\frakH \times 2^\frakH\to 2^\frakH$.
\end{proof}

\begin{proof}[Proof of Proposition \ref{horo=Roller}]
Fix an origin $o\in X$. Let $\xi\in \~ X$ and $x_n \in \~X$ a sequence with $x_n \to \xi$. For $x\in X$, set $m=m(o,x,\xi)$ and observe that $m\in X$. Also set $m_n = m(o,x,x_n)$. By definition of the median, we have:

\begin{eqnarray*}
h_{x_n}(x) &:=& d(x,x_n) -d(o,x_n) \\
&=& d(x,m_n) + d(m_n,x_n) - d(o,m_n) - d(m_n, x_n)\\
&=& d(x,m_n) - d(o,m_n).
\end{eqnarray*}
Taking limits and utilizing Lemma \ref{mediancontinuous} which guarantees the continuity of the median, we deduce
$$h_\xi(x)=d(m,x)-d(m,o).$$ 

Next observe that, since $x\in X$, we have $m \in I(o,x)\subset X$, and hence $h_\xi(x)<+\infty$, that is $h_\xi$ is a function from $X$ to $\R$. It is continuous as the metric is continuous. We denote by $H:\~ X\to \mathcal C(X)$ the map which associates $h_\xi$ to $\xi$. We have shown that $h_{x_n}\to h_\xi$ and from this it is straightforward to conclude that the map $H:\overline X\to X\cup\bdh X$ is continuous.


Let us prove that $H$ is injective. Assume that $\xi,\xi'\in \~ X$ are such that $h_{\xi}=h_{\xi'}$. Let $x$ be a vertex adjacent to $o$ and $k$ be the half-space containing $x$ but not $o$. We have $h_\xi(x)=1$ if $\xi\not\in k$ and $h_\xi(x)=-1$ otherwise. It follows that $\xi\in k$ if and only $\xi'\in k$. The same argument works starting from any vertex (by induction on the distance to $o$). Hence we have $U_\xi=U_{\xi'}$ and therefore $\xi=\xi'$.

Now, let $f$ be a horofunction. Hence $f$ is a limit of functions of the form $(h_{x_n})$, for some sequence $(x_n)$ of vertices. Let $(x_{\f(n)})$ be a subsequence converging to some $\xi\in \overline X$. Then it follows that $(h_{x_{\f(n)}})$ converges to $h_\xi$, hence that $f=h_\xi$. So the map $H$ is surjective, hence bijective. Since $\~ X$ is compact it is a homeomorphism.

Finally, the above arguments show that $H|_{\partial X}$ is a homeomorphism from $\partial X$ to $\bdh X$.
\end{proof}

We also note, for future use, the following cocycle relation:

\begin{lemma}\label{cocycle horo}
Let $g_1,g_2\in\Aut(X)$, $\xi\in\overline X$. 
Then
$$h_\xi(g_2^{-1}g_1^{-1} o) = h_{g_2\xi}(g_1^{-1}o) + h_\xi(g_2^{-1} o).$$
\end{lemma}

\begin{proof}
Let $x_n$ be a sequence converging to $\xi$. Then 
\begin{eqnarray*}
h_{g_2\xi}(g_1^{-1}o) + h_\xi(g_2^{-1} o) 
&=& \lim_{n\to+\infty} d(g_2 x_n,g_1^{-1} o )-d(g_2 x_n,o)+d(x_n,g_2^{-1} o)-d(x_n,o)\\
&=& \lim_{n\to+\infty} d(x_n,g_2^{-1}g_1^{-1} o) - d(x_n,o)\\
&=& h_{\xi}(g_2^{-1}g_1^{-1} o)
\end{eqnarray*}

\end{proof}

The equality of Lemma \ref{cocycle horo} is better understood and remembered in the following form: if $\sigma(g,\xi)=h_\xi(g^{-1}o)$, then we have
$$\sigma(g_1g_2,\xi)=\sigma(g_1,g_2\xi)+\sigma(g_2,\xi).$$
In other words $\sigma$ is an additive cocycle.

\subsection{Remarks on $B(X)$ and $R(X)$}

In \cite{NevoSageev},  Nevo and Sageev introduce another boundary which they call $B(X)$ and define as follows: 

\begin{definition}
A point $\xi\in \partial X$ is called \emph{non-terminating} if for every $h\in U_{\xi}$ there exists $k\in U_{\xi}$ with $k\subset h$.

The set of non-terminating points is denoted by $\partial_{NT}X$. The \emph{non-terminating boundary} $B(X)$ is the closure of $\partial_{NT} X$.
\end{definition}

It is clear that any regular point is non-terminating. Hence, we always have $R(X)\subset B(X)$.

Furthermore, one of the main results of \cite{NevoSageev} is:
\begin{theorem}
Let $\Gamma$ be a group acting properly and cocompactly on the complex $X$. Then the action of $\Gamma$ on $B(X)$ is minimal and strongly proximal.
\end{theorem}

The set $R(X)$ is clearly a $\Gamma$-invariant closed subset of $B(X)$. Thus, the minimality of the action on $B(X)$ implies the following.

\begin{corollary}
Assume that $\Isom(X)$ contains a discrete subgroup acting cocompactly.  Then $B(X)=R(X)$.
\end{corollary}

For concrete examples, such as the Salvetti complex associated to a Right Angled Artin Group, it is straightforward to check the equality $B(X)=R(X)$. Let us do so in a particular case. It is of interest to us as we will modify it in Proposition \ref{B(x) not R(X)} to give an example of a complex where $B(x)\neq R(X)$.

\begin{example}\label{B=R}
Let  $X=X(\Z^2*\Z)$  be the universal cover of the Salvetti complex associated to  $\Z^2*\Z$, where $\Z^2 = \<a,b\>$ and $\Z = \<c\>$ are the generators of the free factors. It is straightforward to check that there are unique points $x_n, x_\8\in B(X)$ such that  
$c^ka^nb^n \to x_n$ as $k\to \8$, and $a^nb^n \to x_\8$  as $n\to \8$.
Furthermore, $x_n\in \partial_r X $ and $x_\8 \in \partial_{NT} X\setminus \partial_r X $. On the other hand, as $n\to \8$ we have $x_n\to x_\8$. Finally,  observe that a similar construction can be applied to any element of  $\partial_{NT} X\setminus \partial_r X $ and so we have that $B(X) = R(X)$.  
\end{example}

\begin{proposition}\label{B(x) not R(X)}
There exists a complex $X$, with $\Isom(X)$ acting essentially and non-elementarily, but with $R(X)\neq B(X)$.

In particular, the action of $\Isom(X)$ on $B(X)$ is not minimal.
\end{proposition}

\begin{proof}
We retain the notation of Example \ref{B=R}. We will construct a subcomplex of $X' \subset X(\Z^2*\Z)$ and it will have an action of $\Z*\Z = \<ab\>*\<c\>$ which is essential and non-elementary. First observe that the action of $ab$ on the plane associated to $\Z^2$ is essential and we have an embedding (which is a similarity) of $\Z \hookrightarrow \Z^2$ by mapping a generator of $\Z$ to $ab$. 
This embedding extends to an embedding of the tree associated to $F_2\cong \Z*\Z$ into $X$. We define $X'$ as the ($\ell^1$-) convex hull of the image of this tree in $X$. In particular, $X'$ contains every plane of $X$ containing an axis of a conjugate of $ab$. It is straightforward to check that since the action of $\<ab\>$ is essential on the plane, the action of $\Z*\Z$ is essential as well. 

Now, the non-terminating points corresponding to $(-\8, \8)$ and $(\8, -\8)$ in the plane containing the axis of $ab$ are isolated in $B(X)$. Since these are not regular, we deduce that $B(X) \neq R(X)$.
\end{proof}

\section{Uniqueness of the Stationary Measure}

Let $X$ be a finite dimensional CAT(0) cube complex, $\Gamma$ a group acting on $X$ and $\mu$ an admissible measure on $\G$. We denote by $B$ the Furstenberg-Poisson boundary of $(\G,\mu)$. Our goal in this section is to prove that there is a unique stationary measure on $\partial X$.

The main tool is the following:

\begin{theorem}\label{eta}
Assume that the action of $\Gamma$ on $X$ is non-elementary and essential.
There is a $\Gamma$-equivariant map $\eta:B\to \partial X$. 

Furthermore, for every such equivariant map and almost every $b\in B$, $\eta(b)\in \partial_r X$.
\end{theorem}

\begin{proof}
The existence of the map is \cite[Theorem 4.1]{CFI} in the symmetric case and \cite[Theorem 7.1]{Fernos} in the general case. The fact that $\eta(b)$ is almost surely regular is \cite[Theorem 7.7]{Fernos}.
\end{proof}





\begin{proposition}\label{UniqueDirac}

If the action of $\Gamma$ on $X$ is non-elementary and essential 
 then there is a unique $\Gamma$-equivariant measurable map $\f:B\to \Prob(\overline X)$ and for almost every $b\in B$, $\f(b)$ is the Dirac mass at $\eta(b)$.
\end{proposition}

\begin{proof}
We first prove the result for $X$ irreducible and then use this to prove the result in general. 

Assume that $X$ is irreducible. Let $B_-$ be the Poisson boundary for the inverse measure $\ch \mu$. Recall from Theorem \ref{IsomErgodic} that $B_-\times B$ is isometrically ergodic. By \cite[Theorem 7.1 \& 7.7]{Fernos}, there is another equivariant map $\eta_-:B_-\to \~X$ with essential image in $\bdr X$. Furthermore it follows from the proof of \cite[Theorem 7.1]{Fernos} that we have, for almost every $(b_-,b)\in B_-\times B$, $\eta_-(b_-)\neq \eta(b)$.

We claim first that there is a unique measurable and $\G$-equivariant map $\eta:B\to \~X$. Indeed, if there is another such map $\eta'$, then by ergodicity we have almost surely $\eta(b)\neq \eta'(b)$, and the same argument as in \cite[Theorem 7.1]{Fernos}  also proves that $\eta_-(b_-)\neq \eta'(b)$ almost surely. 
Now consider the map $p:B_-\times B\to \~X$ defined by $p(b_-,b)=m(\eta_-(b_-),\eta(b),\eta'(b))$. By Lemma \ref{medianinX} it follows that $p(b_-,b)$ is almost  surely in $X$. Obviously $p$ is $\G$-equivariant, and measurable by Lemma \ref{mediancontinuous}. By metric ergodicity, it is constant. Hence $\G$ fixes a point in $X$, contradicting the assumption that the $\G$-action is non-elementary and proving the claim.

Now let $\f: B \to \Prob(\~X)$ be a measurable $\G$-equivariant map and let us show that for almost every $b\in B$, $\f(b)$ is the Dirac mass at $\eta(b)$. We start with some notation, borrowed from \cite{CFI} (see also \cite{Fernos}).
To a measure $\mathrm{m}\in\Prob(\partial X)$, we can associate three subsets of the set of half-spaces: the heavy half-spaces $H_\mathrm{m}^+$ (of measure $>1/2$), the light ones $H_\mathrm{m}^-$ (of measure $<1/2$), and the balanced ones $H_\mathrm{m}$ (of measure 1/2).  It is easy to see that for any measure $\mathrm{m}$, the set $H_\mathrm{m}^+$  is a consistent set of half-spaces and  hence   if there are no balanced half-spaces then there exists an element $x\in \~X$ such that $U_{x} = H_\mathrm{m}^+$, that is $\{x\} = \Cap{h\in H_\mathrm{m}^+}{}h$.

  Assume that $\f(b)$ is not the Dirac mass at $\eta(b)$. Recall that the map which associates to $b$ the intersection of all heavy half-spaces of $\f(b)$, denoted by $H^+_{\f(b)}$  is again measurable and equivariant, so that $H^+_{\f(b)}=U_{\eta(b)}$. Since $\eta(b)$ is a regular point, by Proposition \ref{Rank1Char}, we can find an infinite descending chain of pairwise strongly separated  heavy half-spaces whose intersection is $\eta(b)$. This proves that the measure of $\{\eta(b)\}$ is at least $1/2$.
  
  Hence we can write, for almost every $b$, $\f(b)=\alpha\delta_{\eta(b)}+(1-\alpha)\f'(b)$, with $1/2\leq \alpha <1$, and $\f'(b)(\{\eta(b)\})=0$. Note that $\alpha:B\to [1/2,1)$ is a $\Gamma$-invariant function, so that by ergodicity, it is constant. By assumption we have $\alpha<1$.
 
 The map $\f':B\to \Prob(\overline X)$ is still equivariant.  Now, applying the same argument to $\f'$ instead of $\f$, we conclude that $\f'(b)(\{\eta(b)\})\geq 1/2$. This is a contradiction.

Therefore $\f(b)$ is almost surely a Dirac mass. As was shown above, there is a unique map from $B \to \~X$ and hence $\f(b) = \eta(b)$ almost surely.\\

Now assume that $X$ is a product $X=X_1\times \dots \times X_n$, where each $X_i$ is irreducible. Then there exists a finite index subgroup $\Gamma_0<\Gamma$ which preserves each factor. By \cite[Lemma 2.13]{CFI}, the induced action of $\Gamma_0$ on each factor is again essential and non-elementary. We note also that by Lemma \ref{Poissonfi}, the Furstenberg-Poisson boundary of $(\Gamma,\mu)$ is $\Gamma_0$-equivariantly isomorphic to the Poisson boundary of $(\Gamma_0,\mu_0)$, where $\mu_0$ is the first return probability.

Applying Theorem \ref{eta} to each irreducible factor, we find $\eta_i:B\to\partial X_i$, which in turn gives $\eta: B \to \partial X$, all of which are $\Gamma_0$-equivariant maps. Let $\pi_i:\~X\to \~X_i$ be the projection.  Let $\f: B \to \Prob(\~X)$ be a $\G$ (and hence $\G_0$)-equivariant map. As was shown above in the irreducible case, the $\Gamma_0$-equivariant map $(\pi_i)_*\f(b)$ is equal to the Dirac mass at $\eta_i(b)$. This means that $$\f(b)(\overline {X_1}\times\dots\times \overline{X_{i-1}}\times\{\eta_i(b)\}\times \overline{ X_i}\times\dots\times \overline{X_n})=1.$$ Since this holds for each $i$, we see that indeed $\f(b)(\{\eta(b)\})=1$, meaning that $\f(b)$ is the Dirac mass at $\eta(b)$.

\end{proof}

\begin{cor}\label{stationary}
Assume that the action of $\Gamma$ on $X$ is essential and non-elementary. Then there is a unique stationary measure on $\overline X$.
\end{cor}

\begin{proof}
Apply Proposition \ref{UniqueDirac} and Corollary \ref{cor:uniquestationmeas}.
\end{proof}

\begin{remark}
The assumption that the action is essential cannot be removed. Indeed, take the example of the free group $F_2$ acting on the product $T\times L$, where $T$ is the Cayley tree of $F_2$, and $L$ is a line (with trivial action). Let $\nu$ be the stationary measure on $\partial T$. Then for every $x\in L$, the measure $\nu\times \delta_x$ is a stationary measure on $\partial T\times x\subset \partial(T\times L)$.
\end{remark}

\begin{remark}
During the writing of this paper,  it has been proved in the paper \cite{KarSageev} that, in the irreducible case, the action on $R(X)$ is minimal and strongly proximal. For irreducible spaces, the uniqueness of the stationary measure follows, as explained in \cite{NevoSageev}.
\end{remark}

\section{Convergence to the Roller Boundary}

Now that we understand better the stationary measure on the boundary, we can attack the problem of the convergence of the random walk. Recall that Theorem \ref{eta} guarantees the existence of a measurable and $\G$-equivariant map $\eta:B\to \partial X$. As $B$ is a factor of $\Omega$, we can also consider the composition $\Omega\to B\to \partial X$, which we still denote by $\eta$.
Our goal is to prove:

\begin{theorem}\label{cvg}
Assume that the action of $\Gamma$ on $X$ is non-elementary, essential, and preserves each irreducible factor of $X$. Then for almost every $\omega\in\Omega$ the sequence $(Z_n(\omega)o)_n$ converges to $\eta(\omega)$.
\end{theorem}

Our strategy is inspired by a proof of Kaimanovich in the case of hyperbolic groups \cite[Theorem 2.4]{Kaimanovich}, although we have to face some technical difficulties, these are overcome thanks to the fact that  regular points are well-behaved. This is exemplified by the following:

\begin{proposition}\label{wlimsupp}
Assume $X$ is irreducible, and let $\lambda$ be a non-atomic measure on $\~X$, such that $\lambda(\partial_rX)=1$. If $g_n\in \Gamma$ is such that $g_no \to \xi_0\in\overline X$, and $(g_n\lambda$) weakly converges to $\nu$, then  $\nu(\overline{[\xi_0]})=1$.
\end{proposition}

The proof of Proposition \ref{wlimsupp}, will rely on some more lemmas. 

\begin{lemma}\label{nu=1}
Let $G$ be a group acting by homeomorphisms  on some metrizable compact space $C$,  $(g_n)$ be a sequence in $G$, $\lambda$ a probability measure on $C$ and $A\subset C$ be a Borel subset such that for almost all $x\in C$, any limit point of $(g_nx)$ belongs to $A$.
If  $(g_n\lambda)$ weakly converges to $\nu$ then $\nu(\overline A)=1$.
\end{lemma}

\begin{proof}
As we may replace $A$ with its closure without affecting the hypotheses or conclusion, let us assume that $A$ is closed. Fixing a metric compatible with the topology, denote by $A^\varepsilon$ the  $\varepsilon$-neighborhood of $A$. The assumption implies that for almost every $x\in C$ and every $n$ large enough, we have $g_nx\in A^\varepsilon$: if not, there is a subsequence which avoids  $A^\varepsilon$ completely, and any limit point of this subsequence does not belong to $A$.

We note that, since $ A = \bigcap_{n\in\N}A^{1/n}$, it is sufficient to prove that $\nu(A^\varepsilon)=1$ for all $\varepsilon>0$.

Fix $\varepsilon>0$. Denote by $\~{A^{\varepsilon}}$ the closure of the $\varepsilon$ neighborhood of $A$. By Urysohn's Lemma, there is a continuous $f: C\to [0,1]$ bounded above and below by the characteristic functions
$$\1_{\~{A^{\varepsilon}}}(x)\leq f(x) \leq \1_{A^{2\varepsilon}}(x).$$
 By assumption, for each $x$ there is an $n$ sufficiently large such that $g_nx\in A^\varepsilon$ and hence $f_n(x):=f(g_nx)\to 1$. It follows from the Dominated Convergence Theorem that $\nu(f)=1$. Hence $\nu(A^\varepsilon)=1$, which concludes the proof of the lemma.

\end{proof}

\begin{lemma}\label{liminX}
Let $\xi_0\in\overline X$ and $g_n\in\Gamma$ be such that $g_n o\to \xi_0$. Then for every $x\in X$ all limit points of the sequence $(g_n x)_n$ belong to $ [\xi_0]$.
\end{lemma}

\begin{proof}

We have for all $n$, $|U_{g_no}\triangle U_{g_n x}|=|U_o\triangle U_x|$. Let $a$ be a limit point of $(g_n x)$. If $h_1,\dots, h_k$ are half-spaces in $U_a\triangle U_{\xi_0}$ then we see that for  $n$ large enough we have $h_i\in U_{g_no}\triangle U_{g_n x}$ for all $1\leq i\leq k$. Hence we have $k\leq |U_o\triangle U_x|$. In other words, there are at most $|U_o\triangle U_x|$ half-spaces in $U_a\triangle U_{\xi_0}$. This means that $a\in [\xi_0]$.
\end{proof}

The previous lemma can be extended to the convergence of  points in the Roller boundary, up to passing to a subsequence and excluding finitely many points:

\begin{lemma}\label{limitpoints}
Let $g_n\in \G$. If there is $\xi_0 \in \~X$ and $o \in X$ such that $g_n o \to \xi_0$ then there is a subsequence $\f(n)$ and $\xi_1, \dots, \xi_k$ such that if $\xi \in \partial_r X \setminus\{\xi_1, \dots, \xi_k\}$ then all limit points of $(g_{\f(n)}\xi)$ belong to $\~{[\xi_0]}$.
\end{lemma}

\begin{proof}
Observe that if $\xi_0 \in X$ then the result follows as $[\xi_0]=X$. Therefore, assume $\xi_0 \in \partial X$. 

 Let $\{h^1_m: m\in \N\}, \dots, \{h^k_m: m\in \N\}$ be the descending chains provided by Lemma \ref{lemma finitely many chains}, i.e. such that $\~{[\xi_0]} = \overset{k}{\Cap{i=1}}\Cap{m\in \N} h^i_m$.
 
Then, the following dichotomy holds:  either  for every $\xi \in \partial_r X$ we have that all limit points of the sequence $(g_n \xi)$ belong to $ \Cap{m}h_m^1$ or there is an $\xi_1\in\partial_r X$ and a subsequence $\f_1(n)$ for which $g_{\f_1(n)}(\xi_1) \to \~\xi_1\notin \Cap{m}h_m^1$. In case all limit points belong to $\Cap{m}h_m^1$ we set $\f_1(n) = n$, and define $\xi_1$ arbitrarily.

By the same process, we construct inductively, for each $1<i\leq k$,  a subsequence $\f_i$ of $\f_{i-1}$, and $\xi_i\in\bdr X$, such that, for every $j\leq i$, we have

\begin{itemize}
\item[(a)] either the limit points of $(g_{\f_i(n)}\xi)_n$ are in $\Cap{m\geq 0} h^j_m$ (and we define $\xi_j$ arbitrarily)
\item[(b)] or we find $\xi_j\in\bdr X$ and $\~\xi_j\not\in \Cap{m}h_m^j$ with $g_{\f_i(n)}\xi_j\to\~\xi_j$. 
\end{itemize}
%
%

Fix $i\leq k$ and for simplicity let $\f(n)=\f_k(n)$. Let us now show that for every $\xi \in \partial_r X \setminus\{\xi_i\}$ we have that the limit points of $g_{\f(n)}\xi $ belong to $\Cap {m}{}h_m^i$.
If $\xi_i$ was chosen arbitrarily as in case (a) above then there is nothing to prove. Therefore, up to passing to a subsequence, assume that $g_{\f(n)}\xi_i\to\~\xi_i\notin \Cap{m}{}h_m^i$.

Consider  $S(\xi, \xi_i) = \I(\xi, \xi_i)\cap X$ (which is not empty by Lemma \ref{IinterX}).
 Let $x\in S(\xi, \xi_i)$. Since $x$ is at finite distance from $o$, it follows from Lemma \ref{liminX} that every limit point of $(g_{\f(n)}(x))$ belongs to $\~{[\xi_0]}$. So, for each $m\in \N$ there is an $N_i$ so that if $n>N_i$ then  $g_{\f(n)}(x) \in h_m^i$. 
On the other hand, since $\~\xi_i\notin \Cap{m}h_m^i$,  there exists $M_i$ such that  $\~\xi_i \notin h_m^i$ for all $m>M_i$. Since $g_{\f(n)}\xi_i$ converges to $\~\xi_i$ there is an $N'_i>N_i$ so that if $n>N'_i$ and $m>M_i$ then $g_{\f(n)}(\xi_i)\notin h_m^i$. 

Fix $m>M_i$. If we had $g_{\f(n)}\xi \not\in h_m^i$, then by convexity of the interval $\I(g_{\f(n)}\xi,g_{\f(n)}\xi_i)$ we would have $g_{\f(n)}x\not\in h_m^i$. So for $n>N'_i$ we have $g_{\f(n)}\xi\in h^i_m$. 

Now redefine $\xi \in \partial_r X\setminus\{\xi_1, \dots, \xi_k\}$ and let $\overline \xi$ be a limit point of the sequence $(g_{\f(n)}\xi)$. The above argument shows that  
 $\~\xi\in \Cap{m}h_m^i$, for each $i = 1, \dots, k$, i.e. 
$$
\~\xi \in \overset{k}{\Cap{i=1}}\Cap{m\in \N} h^i_m = \~{[\xi_0]}.
$$
\end{proof}

\begin{proof}[Proof of Proposition \ref{wlimsupp}]

We first replace as we may $(g_n)$ by a subsequence satisfying the conclusion of Lemma \ref{limitpoints}. Since $\lambda$ is non atomic, we have that for $\lambda$-almost every $\xi$, every limit point of $(g_n \xi)$ is in $\overline{[\xi_0]}$. By Lemma \ref{nu=1}, this implies that $\nu(\overline{[\xi_0]})=1$.
\end{proof}

\begin{lemma}\label{cvgmes}
 Assume $X$ is irreducible. Let $\lambda$ be a non-atomic measure on $\partial_r X$. Let $g_n\in G$ be such that $g_n\lambda$ weakly converges to a Dirac mass $\delta_b$, for some $b\in \partial_rX$. Then $g_n o$ converges to $b$.
\end{lemma}

\begin{proof}
Let $b'$ be a limit point of $(g_no)$. 
By Lemma \ref{wlimsupp} we have that $\delta_b$ is supported on $\overline{[b']}$.
 Hence $b\in \overline{[b']}$. By Lemma \ref{existschain}, there exists a sequence of half-spaces $(h_n)$ such that $[b']\subset \bigcap_{n\in\N} h_n$. Since every half-space is closed by definition, we have $\overline{[b']}\subset \bigcap_{n\in\N} h_n$. Hence $b\in\bigcap_{n\in\N} h_n$. Since $b\in\partial_r X$, by Proposition \ref{singleton}, it follows that $\bigcap_{n\in\N} h_n=\{b\}$. So $b'=b$.

\end{proof}

\begin{proof}[Proof of Theorem \ref{cvg}]
Let $\lambda$ be the unique stationary measure on $\overline X$. As we assume the action is non-elementary, the measure $\lambda$ is not atomic. Assume first that $X$ is irreducible. We know that $Z_n\lambda$ converges to the Dirac mass $\delta_{\eta(\omega)}$, where $\eta(\omega)\in\partial_rX$ almost surely. By Lemma \ref{cvgmes}, it follows that $Z_n o$ converges to $\eta(\omega)$.

Now if $X$ is not irreducible, but $\Gamma$ preserves each factor $X_i$ of $X$, then the action of $\Gamma$ on $X_i$ is still non-elementary and essential, and the previous argument proves that the projection of $Z_no$ to $X_i$ converges to some point in $\partial X_i$. Hence $Z_no$ also converges to a point in the boundary of $X$.
\end{proof}

\section{Positivity of the Drift}\label{drift}

\subsection{The Drift}

Before getting into the specifics of our situation, we recall some basic general facts about the drift of an action. Assume that $\G$ acts on a metric space $X$. Choose a vertex $o$ in $X$. This gives rise to a seminorm on $\G$ defined by $|g|=d(go,o)$. The \emph{drift} relative to $|\cdot|$ is defined as follows.

\begin{definition}
The drift of the $\mu$-random walk with respect to a seminorm $|\cdot|$ is:
$$
\lambda= \underset{n}{\inf} \frac{1}{n}\int_\Omega |Z_n(\omega)|\,d\Ps(\omega).
$$
\end{definition}

The following is a standard application of Kingman's Subbaditive Ergodic Theorem:

\begin{theorem}
For almost every $\omega\in\Omega$ we have 
$$\lambda = \Lim{n\to \8}\frac{1}{n} |Z_n(\omega)|.$$
\end{theorem}

Furthermore, $\lambda$ is finite whenever $\mu$  
 has finite first moment (with respect to $|\cdot|$), i.e. $\Sum{g\in \G} \mu(g)|g|<\8$.

\subsection{Proof of the Positivity of the Drift}

Our goal in this section is to prove that the speed at which the random walk goes to infinity is always linear. Our proof follows a classical strategy 
which was initiated by Guivarc'h and Raugi for linear groups \cite{GR85}. Ledrappier extended it to free nonabelian groups \cite{Led01}, and Benoist and Quint to Gromov hyperbolic groups \cite{BenoistQuintHyp}. 

The main aim of this section is to prove the following:

\begin{theorem}\label{theorem pos drift}
 Let $\G \to \Aut(X)$ be an essential and nonelementary action, $\mu$ a probability measure on $\G$, $o\in X$ such that $\mu$ has finite first moment with respect to $|\cdot|$. Then $\lambda >0$. 
\end{theorem}

Recall from section \ref{sec:horo}  that $\partial X$ is isomorphic to the horofunction boundary of $X$ with the combinatorial distance. If $\alpha\in\partial X$, we denote $h_\alpha$ the corresponding horofunction.

We denote by $\nu$ the unique stationary measure on $\partial X$. By Theorem \ref{eta}, we have $\nu(\partial_r X)=1$. 

The positivity of the drift will follow easily once we prove the following:

\begin{prop}\label{prop distance follows horofunction}
Assume that $\G$ stabilizes each factor of $X$. Then for every $\xi\in\partial X$ and $\Ps$-almost every $\omega\in\Omega$, there exists $C>0$ such that for all $n>0$ we have
$$|d(Z_n(\omega)o,o)-h_\xi(Z_n(\omega)o)|<C.$$
\end{prop}

\begin{proof}

First, we claim that the conclusion of the proposition does not depend on the choice of the basepoint $o$. Indeed, assume that  $$|d(Z_n(\omega)o,o)-h_\xi(Z_n(\omega)o)|<C.$$ If $o'$ is another basepoint then  $$d(Z_n o',o')\leq d(Z_no',Z_no)+d(Z_no,o)+d(o,o'),$$ and hence $d(Z_n o',o')-d(Z_n o,o)\leq 2d(o,o')$.  By symmetry 
$$\vert d(Z_n o',o')-d(Z_n o,o)\vert\leq 2d(o,o').$$
 Similarly $\vert h_\xi(Z_n o)-h_\xi(Z_n o')\vert \leq 2d(o,o')$.
Hence
 $$\vert d(Z_n(\omega)o',o')-h_\xi(Z_n(\omega)o') \vert \leq 4 d(o,o')+C,$$
 which proves the claim.

Let $\xi\in\partial X$. By Theorem \ref{cvg}, for $\Ps$-a.e. $\omega\in \Omega$, there is $\eta(\omega)\in \partial_r X$ such that $Z_n(\omega)o\to\eta(\omega)$ for every $o\in X$. As the action is non-elementary, we know that $\eta(\omega)\neq \xi$ almost surely. Fix such a generic $\omega$ and set $\eta= \eta(\omega)$ and $Z_n = Z_n(\omega)$. By the claim above and Lemma \ref{IinterX}, we may and shall assume
that $o\in I(\eta, \xi)\cap X$.

Recall from Proposition \ref{horo=Roller} that, the median  $m(\xi,x,o)\in\I(x,o)$ is such that
$$h_\xi(x) = d(m(\xi,x,o), x)-d(m(\xi,x,o),o).$$

Let  $m_n = m(\xi,  Z_no, o)$, so that $h_\xi(Z_n o) = d(m_n, Z_n o)-d(m_n,o)$. Then:
\begin{eqnarray*}
d(Z_no,o)-h_\xi(Z_no) & = & d(Z_no,m_n) +d(m_n,o)-\(d(m_n, Z_n o)-d(m_n,o)\)\\
&=&2d(m_n,o)
\end{eqnarray*}
Again, by continuity of the median map, we have that $m_n \to m(\xi,\eta(\omega),o)$.  Recall that we have chosen $o = m(\xi,\eta,o)\in I(\eta, \xi)\cap X$, which is locally compact. Therefore, for $n$ sufficiently large, 
$$d(Z_no,o)-h_\xi(Z_no) =2d(m(\xi,\eta(\omega),o), o)=0.$$

\end{proof}

We immediately deduce that:

\begin{cor}\label{LambdaLimit}
 For every $o\in X$, $\Ps$-a.e. $(Z_n)\in \Omega$ and every $\xi \in \partial X$   we have that  
 $$\lambda = \Lim{n\to \8}\;\frac{1}{n}h_\xi(Z_n o).$$
\end{cor}

Our aim now is to apply results about additive cocycles to our situation. 
%
 To this end, let 
$T: \Omega\times \~X \to\Omega\times \~X $ be defined by $$T(\omega , \xi) = (S\omega, \omega_0^{-1} \xi),$$
where $\omega=(\omega_0,\omega_1,\dots,)$ and $S:\omega\mapsto (\omega_1,\omega_2,\dots)$ is the usual shift.

The following lemma is borrowed from \cite[Proposition 1.14]{BenoistQuintBook}. We include a proof for completeness.

\begin{lemma}\label{T ergodic}
The transformation $T$ preserves the measure $\Ps \times \ch\nu$ and acts ergodically.
\end{lemma}

\begin{proof}
Let $\beta=\Ps\times\ch\nu$. We begin by checking the invariance of $\beta$. Let $\psi$ be a bounded Borel function on $\Omega\times\~X$. Let $\f(x)=\int \psi(\omega,x) d\Ps(\omega)$. By definition we have $\beta(\psi)=\ch\nu(\f)$. On the other hand we get $\beta(\psi\circ T)=\int \psi(S \omega,\omega_0^{-1} x )d\Ps(\omega) d\ch\nu(x)=\ch\nu(\f)$ by stationarity of $\ch\nu$. The invariance of $\beta$ follows.

Now let us turn to the proof of ergodicity of $\beta$.
Let $P$ be the averaging operator relative to $\ch\mu$: if $f$ is a bounded Borel function on $\overline X$, then $Pf(x)=\int f(gx) d\ch\mu(g)$. A measure is $\ch\mu$-stationary if and only if it is $P$-invariant. By Corollary \ref{stationary}, 
the measure $\ch\nu$ is the unique $\ch\mu$-stationary measure on $\~X$. It follows that $\ch\nu$ is $P$-ergodic.

Let $\psi$ be a bounded Borel function on $\Omega\times \~X$ which is $T$-invariant. We have to prove it is constant. Let again $\f$ denote the function defined on $\~X$ by $\f(x)=\int \psi(\omega,x) d\Ps(\omega)$.

We first see that 
$$P\f(x)=\int \psi(\omega,g^{-1}x)d\Ps(\omega)d\mu(g)=\int (\psi\circ T)(\omega,x)d\Ps(\omega)= \f(x)$$
so that $\f$ is $P$-invariant. By the above remark it is constant, say equal to $c$.

Let $\mathscr X_n$ be the sigma algebra generated by the first $n$ coordinates $\omega_0,\dots,\omega_{n-1}$ on $\Omega$ and by the variable $x\in\~X$. Let $\f_n=\EE(\f \mid \mathscr X_n)$. Then we have
\begin{align*}
\f_n(\omega_0,\dots,\omega_{n-1},x)&=\int \psi((\omega_0,\dots,\omega_{n-1},\omega),x)d\Ps(\omega)\\
&=\int \psi\circ T^n((\omega_0,\dots,\omega_{n-1},\omega),x)d\Ps(\omega)\\
&=\int \psi(\omega,\omega_{n-1}^{-1}\dots,\omega_0^{-1} x)d\Ps(\omega)\\
&=\f(\omega_{n-1}^{-1}\dots,\omega_0^{-1} x)\\
&=c
\end{align*}

Since the sequence $(\f_n)$ converges to $\psi$, it follows that $\psi$ is also constant, equal to $c$.
\end{proof}

\begin{proof}[Proof of Theorem \ref{theorem pos drift}]

Assume first that the group stabilizes each factor.
Define the function $F: \Omega\times \~X \to\R$ as 
$$F((\omega_n)_n, \xi) = h_\xi(\omega_0 o)$$
and observe that its value only depends on the first coordinate of $(\omega_n)_n$. For every $\xi\in \bd X$, the function $h_\xi$ is 1-Lipschitz on $X$, so that $| F((\omega_n)_n, \xi) |\leq d(o,\omega_0 o)$. It follows that $\int | F(\omega, \xi) | \,\Ps(\omega) d\ch\nu(\xi)<+\infty$.

Recall from Lemma \ref{cocycle horo}  horofunctions satisfy the following relation:
$$h_\xi(g_2^{-1}g_1^{-1} x) = h_{g_2\xi}(g_1^{-1}x) + h_\xi(g_2^{-1} x).$$

Inductively, this shows that if $Z_k = \omega_1\cdots \omega_k$ (and $Z_0=e$) then 
$$h_\xi(Z_n o) =  \overset{n}{\Sum{k=1}} h_{Z_{k-1}^{-1}\xi}(\omega_k o)$$

Therefore, we have the following calculation:

\begin{eqnarray*}
\frac{1}{n}h_\xi(Z_n o) &=& \frac{1}{n}\;\overset{n}{\Sum{k=1}} h_{Z_{k-1}^{-1}\xi}(\omega_k o)\\
&=& \frac{1}{n}\overset{n}{\Sum{k=1}}F(T^k((\omega_n)_n, \xi)
\end{eqnarray*}

Now, assume that $\mu$ has finite first moment.  
By Proposition \ref{prop distance follows horofunction}, we have that  $\frac{1}{n}h_\xi(Z_n o) \to \lambda$. Thanks to Lemma \ref{T ergodic}, we know that $T$ preserves $\Ps \times \ch\nu$ and is ergodic and so we may apply the Birkhoff Ergodic Theorem and conclude:
$$\frac{1}{n}\overset{n}{\Sum{k=0}}F(T^k((\omega_n)_n, \xi)) \to \int F(\omega, \xi) \,\Ps(\omega) d\ch\nu(\xi).$$
Recall that by Proposition \ref{prop distance follows horofunction}, we know that $|d(Z_n(\omega)o,o)-h_\xi(Z_n(\omega)o)|$ is almost surely uniformly bounded. This together with Theorem \ref{cvg} which guarantees the almost sure convergence of the random walk to the boundary, implies that $h_\xi(Z_no)$ tends to $+\infty$ almost surely. 
This means that $\overset{n}{\Sum{k=0}}F(T^k((\omega_n)_n, \xi)) $ is a transient cocycle in the sense of \cite{Atk76} and hence by Atkinson's Lemma $\int F(\omega, \xi) \,\Ps(\omega) d\ch\nu(\xi)$ is strictly positive \cite{Atk76}.  (See also \cite[Lemma 3.6]{GR85}.)

If the group $\Gamma$ does not stabilize each factor, let $\Gamma_0\lhd\Gamma$ be the finite index subgroup which does. Let $(Z_{\f(n)})$ be the subsequence of the random walk formed by the elements which are in $\Gamma_0$. This is a random walk on $\Gamma_0$, which still has finite first moment by \cite[Lemma 2.3]{Kaima91}. Then by the previous result we have $\frac{Z_{\f(n)}}{n}\to \lambda_0>0$. 

Since we already know that $\frac{Z_n}{n}$ converges, the result follows from the fact that $\frac{\f(n)}{n}$ has a positive limit, which is Lemma \ref{phi(n)/n} below.

\end{proof}

\begin{lemma}\label{phi(n)/n}
Let $\Gamma_0\lhd \Gamma$ be a finite index normal subgroup. Let $(Z_{\f(n)})$ be the subsequence formed by all elements of the random walk which are in $\Gamma_0$. Then there is $C>0$ such that  $\frac{\f(n)}{n}\to C$ almost surely.
\end{lemma}

\begin{proof}
Note first that $\Gamma_0$ is of finite index so it is a recurrent set. Consider the induced random walk on the finite group $\Gamma/\Gamma_0$. It is an irreducible Markov chain. Let $\pi$ be the stationary measure on $\Gamma/\Gamma_0$.

For $n\geq 0$, let $\tau_n=\f(n+1)-\f(n)$. Then $\tau_n$ is a random variable whose law is the law of the first return time to $\Gamma_0$. The expectation of $\tau_n$ is equal to $C:=\frac{1}{\pi(e\Gamma_0)}$.
Furthermore, the $\tau_n$ are independent. By the Law of Large Numbers, we have almost surely $\lim\limits_{n\to +\infty} \frac{1}{n}\sum\limits_{k=0}^{n-1}\tau_n=C$. In other words, $\frac{\f(n)}{n}\to C$.
\end{proof}

\begin{remark}
 Let $d'$ be the CAT(0) metric on $X$ and fix a $\G$-action that is essential and non-elementary. Recall that $d$ and $d'$ are quasi-isometric. So $\mu$ has finite first moment with respect to $d$ if and only if it has finite first moment with respect to $d'$. Theorem  \ref{theorem pos drift} then also shows that  if $\mu$ has finite first moment then the drift with respect to either metric is positive.
\end{remark}

\section{Random Walks and the Visual Boundary}\label{RW on visual}\label{RW and Visual}

\subsection{Convergence to the Visual Boundary}

In this section, we are interested in the almost sure convergence to the visual boundary. Karlsson and Margulis showed that  if $\mu$ has finite first moment and if the drift is positive, then almost surely there is an $\xi\in\bd X$ such that $Z_n o$ converges to $\xi$ \cite{KarlssonMargulis}.  
We aim to improve on this by getting rid of these conditions. An important tool in our proof will be the notion of a squeezing point which was developed in Section \ref{Squeezing}.

\begin{theorem}\label{thm:cvgxi}
There exists a map $\xi:B\to \bd X$ such that, for all $o\in X$, almost surely  $Z_n(\omega) o$ converges to $\xi(\omega)$. Furthermore $\xi(\omega)$ is almost surely a squeezing point.
\end{theorem}

We will require:

\begin{proposition}\label{asspecial} 
Almost surely, the point $\eta=\lim\limits_n Z_n o$ is a squeezing point of the Roller boundary.
\end{proposition}

The proof will use the following useful lemma, proved in \cite[Lemma 7.8]{Fernos}. 

\begin{lemma}\label{lem:Sinfinite}
Let $X$ be irreducible with an essential and nonelementary action of $\G$. Let $\nu$ (resp. $\check\nu$) be the stationary measure on $\overline X$ for the measure $\mu$ (resp. $\check \mu$). 

Let $\mathcal S\subset \frakH^2$ be a non-empty, $\Gamma$-invariant set, with $h\subset k$ for every $(h,k)\in \mathcal S$. Then for $\check\nu\otimes \nu$-almost every $(\eta_-,\eta_+)$, the set of $(h,k)\in\mathcal S$ with $\{h,k\}\subset U_{\eta_+}\setminus U_{\eta_-}$ is infinite.
\end{lemma}

\begin{proof}[Proof of Proposition \ref{asspecial}]
It suffices to treat the case of an irreducible complex, so we assume that $X$ is irreducible. Then we know that there exists some pair of super strongly separated
 half-spaces. Fix such a pair, and let $r$ be the distance between these two half-spaces. 

Let $\mathcal S$ be the set of pairs $\{h,k\}$ where $h\subset k$ are super strongly separated half-spaces at distance $r$. Then $\mathcal S$ is a non-empty, $\Gamma$-invariant collection of half-spaces. By Lemma \ref{lem:Sinfinite}, for $\check{\nu}\otimes \nu$-almost every $(\eta',\eta)$,  the set of pairs $(h,k)\in \mathcal S$ such that $h$ and $k$ contain $\eta$ but not $\eta'$ is infinite. Let $\mathcal S(\eta,\eta')$ be the set of all such $(h,k)$. 

We claim that for every $(h,k)\in\mathcal S(\eta,\eta')$, there exists $(h',k')\in\mathcal S(\eta,\eta')$ such that $k'\subset h$ (so that we have $h'\subset k'\subset h\subset k$). Indeed, if it were not the case, then this would mean that there is a collection of half-spaces $h_i$ such that there is some $k_i$ with $(h_i,k_i)\in \mathcal S(\eta,\eta')$ and which are minimal (for inclusion) with this property. By minimality, the half-spaces $h_i$ are all transverse, so there can only be at most $N$ of them, where $N$ is the dimension of $X$. This means that we get a map from $B_-\times B$ to the countable set $\frakH^N$. By isometric ergodicity (Theorem \ref{IsomErgodic}), this map must be essentially constant. Hence there is a finite family of walls which is $\Gamma$-invariant, contradicting the assumptions on the action.

Now fix a generic pair $(\eta,\eta')$ as above. Fix $(h_0,k_0)\in \mathcal S(\eta,\eta')$ and extend as above to a decreasing sequence $h_{n+1}\subset k_{n+1}\subset h_{n}\subset k_{n}$. Letting $x\in I(\eta,\eta')\cap X$ such that $x\in h_0^*\cap k_0^*$ shows that $\eta$ is squeezing.

\end{proof}

\begin{proof}[Proof of Theorem \ref{thm:cvgxi}]
We know by Proposition \ref{asspecial} that $Z_no$ converges to a squeezing point of the Roller boundary. So by Lemma \ref{lem:specialvisual} we deduce that there is some $\xi\in\bd X$ such that $Z_n o$ converges to $\xi$. 
\end{proof}

If $X=X_1\times\dots\times X_n$ is reducible, the situation is  different. In that case there is no point $\eta\in\partial X$ such that the set $Q(\eta)$ is reduced to a singleton. Indeed, if $\eta_i\in\bdr X_i$ then $Q((\eta_1,\dots,\eta_n))$ is a sector in the sphere $S^{n-1}$. A point in this sector can be represented by a half-line generated by some vector $(\lambda_1,\dots,\lambda_n)$, where $\lambda_i>0$.

\begin{theorem}\label{cvg:reducible}
Let $\G$ act on $X=X_1\times\dots\times X_n$ non-elementarily, essentially, and preserving each irreducible factor $X_i$. Assume also that $\mu$ has finite first moment. Let $\lambda_i$ be the drift for the action of $\G$ on $X_i$, and let $\eta_i$ be the limit of $Z_n o$ in the factor $X_i$.

Then almost surely the limit of $Z_n o$ in $X$ is the point of $Q(\eta_1,\dots,\eta_n)$ corresponding to the vector $(\lambda_1,\dots,\lambda_n)$.
\end{theorem}

\begin{proof}
Let $d_i$ be the CAT(0) metric on the factor $X_i$. We have $d'=\sqrt{\sum\limits_{i=1}^n d_i^2}$. 
Note that the measure $\mu$ still has finite first moment for the action on each factor. Let $\Lambda=\sqrt{\sum\limits_{i=1}^n\lambda_i^2}$ be the drift on $X$.
 Using \cite{KarlssonMargulis}, we get ``sublinear tracking": almost surely, there exists a geodesic ray (for the CAT(0) metric) $g_i$ in $X_i$ such that $\frac{d(g(\lambda _i n),Z_n o_i)}n$ tends to $0$.

 Now consider the quadrant defined by the geodesic rays $g_1,\dots,g_n$. A point in this quadrant is of the form $(g_1(t_1),\dots,g_n(t_n))$ with $(t_1,\dots,t_n)\in(\R^+)^n$. Let $g(t)=(g_1(\lambda_1 t),\dots,g_n(\lambda_n t))$. Then we have  that
  \begin{align*}
 d'(g(t),g(s))&= \left(\sum_{i=1}^n d_i(g_i(\lambda_i t),g_i(\lambda_i s))^2\right)^{1/2}\\
 &=\left(\sum_{i=1}^n \lambda_i^2 (t-s)^2\right)^{1/2}\\
 &=\Lambda |t-s|
\end{align*} 

  In other words, 
$g$ is a geodesic ray, travelled at speed $\Lambda= \sum \lambda_i^2$. Its endpoint is exactly the point of $Q(\eta_1,\dots,\eta_n)$ corresponding to the vector $(\lambda_1,\dots,\lambda_n)$. Furthermore we see easily that 
$$\lim\limits_{n\to +\infty} \frac{d'
(g(t),Z_n o)}{n}=0.$$
It follows that $Z_n o $ converges to the point of $\bd X$ corresponding to $g$.
\end{proof}

\subsection{Uniqueness of the Stationary Measure}
We first note that, under our assumptions, the visual boundary is a Polish space. Indeed the visual boundary is obtained as an inductive limit of balls centered at a fixed origin $o$ \cite[II.8.5]{Bridson_Haefliger}. A complete metric can be described as follows: the distance between two fixed geodesic rays $\rho$ and $\rho'$ starting from $o$ is $\delta(\rho,\rho')=\sum_{n=1}^{+\infty}2^{-n}d(\rho(n),\rho'(n))$.  This allows us to use 
 Corollary \ref{cor:uniquestationmeas} to reduce the problem of uniqueness of the stationary measure on $\partial_\eye X$ to the uniqueness of a $\G$-equivariant measurable $B \to \Prob(\partial_\eye X)$. 
 

\begin{theorem}\label{thm:uniquestatbd}
Assume that $\G$ is a group with a non-elementary, essential action on an irreducible complex $X$. Then there is a unique stationary measure on $\partial_\eye X$. 
\end{theorem}

Theorem \ref{thm:uniquestatbd} then follows from Corollary \ref{cor:uniquestationmeas} together with Lemma \ref{lem:uniquemapbd}. 
 To this end, we will need:

\begin{lemma}\label{Visual Projection to X}
 Let $\xi_-, \xi_+\in \partial_\eye X$ be distinct squeezing points. Then there is a map $$\f_{\xi_-,\xi_+}:\partial_\eye X \setminus \{\xi_-, \xi_+\} \to 
 X,$$
 such that for every $g\in\Gamma$, we have $g\f_{\xi_-,\xi_+}(\xi_+)=\f_{g\xi_-,g\xi_+}(g\xi)$.
\end{lemma}

\begin{proof}
 Fix distinct squeezing points $\xi_-, \xi_+\in \partial_\eye X$. By Lemmas \ref{lem:specialvisual} and \ref{lem: visaula and roller squeezing bijection}, there is a bi-infinite decreasing sequence of pairwise strongly separated half-spaces such  that $\bigcap_{n\in\Z} \bd s_n=\{\xi_-\}$ and $\bigcap_{n\in\Z} \bd s_n^*=\{\xi_+\}$.  Let $\eta\in \partial_\eye X\setminus \{\xi_-, \xi_+\}$. It follows that there exists some $n$ such that $\eta\in \bd s_n^*\cap \bd s_{-n}$. Up to deleting finitely many elements of the sequence, we may and shall assume that $\eta$ is in  $\bd s_0^*\cap \bd s_{-1}$. Let us fix a base vertex $o\in s_0^*\cap s_{-1} \cap \I(\xi_-, \xi_+)\cap X$ and vertices $x_n\in s_0^*\cap s_{-1}\cap X$  so that the $\ell^2$-geodesics between $o$ and $x_n$ converge  to $\eta$.

  Recall that one can associate to $\eta\in \partial_\eye X$  a horofunction relative to the $\ell^2$ metric, which we denote by $b_\eta^{(2)}$. Our goal is to show that $b_\eta^{(2)}|_{\I(\xi_-, \xi_+)\cap X}$ attains a minimum, and the set of points on which this function is minimal is a bounded convex set. The image $\f_{\xi_-,\xi_+}(\eta)$ is then defined to be the center of this set.
 
Let $y_{k}  \in s_{k}^*\cap s_{k-1}\cap  \I(\xi_-, \xi_+)\cap X$. We claim that if $D$ is the dimension of $X$ then 
  $$b_\eta^{(2)}(y_k) \geq \sqrt{D}\cdot  |k|.$$
 
Assume $k>0$. Observing that $o = m(y_k,o,x_n)$ we see:

$$b_{x_n}(y_k) := d(y_k, x_n) - d(o, x_n) = d(y_k,o)\geq k.$$

Recalling the fact that $d'\leq d\leq \sqrt{D} d'$ (where $d$ and $d'$ are the $\ell^1$ and $\ell^2$-metrics respectively) we deduce that $b_{x_n}^{(2)}(y_k) \geq\sqrt{D}\cdot k  $. Taking the limit as $n\to\8$ we get that 
$$b_\eta^{(2)}(y_k) \geq\sqrt{D}\cdot k >0.$$

Observing that $b_\eta^{(2)}(o)=0$ this  shows that the inverse image  of $(-\8,0]$ by the function $b_\eta^{(2)}|_{\I(\xi_-,\xi_+)}$ is non-empty and contained in the bounded convex set 
$s_1^*\cap s_{-1}\cap \I(\xi_-, \xi_+)\cap X$. Hence it has a unique center.  

Finally, the $\G$-equivariance of these projections follows from the equivariance of the horofunctions and the construction of the center.
\end{proof}

Recall that Theorem \ref{thm:cvgxi} gives the existence of a measurable $\G$-equivariant map $\xi_\pm: B_\pm\to \partial_\eye X$.

\begin{lemma}\label{lem:uniquemapbd}

Assume that $\G$ is a group with a non-elementary, essential action on an irreducible complex $X$. There is a unique $\G$-equivariant map $B_+\to \Prob(\bd X)$, which is the map $\omega\mapsto \delta_{\xi_+(\omega)}$.
\end{lemma}

\begin{proof}
Consider the $\xi_-$ and $\xi_+$-pushforward of the measures on $B_-$ and $B_+$ to $\partial_\eye X$. We will call them $\nu_-$ and $\nu_+$, respectively. Recall that they are  $\check{\mu}$ and $\mu$-stationary respectively.

Let $\omega_+\mapsto \nu_{\omega_+}$ be some $\G$-equivariant map from $B_+$ to $\Prob(\bd X)$. By ergodicity, if $\nu_{\omega_+}\neq \delta_{\xi_+(\omega_+)}$ on a positive measure set, then this set has full measure. 

 So assume that we have almost surely $\nu_{\omega_+}\neq \delta_{\omega_+}$. The function $\omega_+\mapsto \nu_{\omega_+}(\{\xi_+(\omega_+)\})$ is $\G$-invariant and hence constant. If $\nu_{\omega_+}(\{\xi_+(\omega_+)\})=\a>0$ then we can define $\nu'=\nu-\a\nu_+$. After renormalization this is a new stationary probability measure such that $\nu'_{\omega_+}(\{\xi_+(\omega_+)\})=0$. So we may and shall assume that $\nu_{\omega_+}(\{\xi_+(\omega_+)\})=0$ for almost every $\omega_+\in B_+$.
 
 We claim next that $\nu_{\omega_+}(\{\xi_-(\omega_-)\})=0$ for almost every $(\omega_-, \omega_+)\in B_+\times B_-$. Indeed, for a fixed $\omega_+$, the measure 
 $\nu_{\omega_+}$ has countably many atoms, so that for $\nu_-$-a.e. $\omega_-\in B_-$  we have that $\nu_{\omega_+}(\xi_-(\omega_-)) = 0$. By Fubini it follows that $\nu_{\omega_+}(\{\xi_-(\omega_-)\})=0$ almost surely.

Theorem \ref{thm:cvgxi} assures us that $\xi_+$ and $\xi_-$ are squeezing points  almost surely. 
Now, apply the projection from Lemma \ref{Visual Projection to X} to obtain for almost every $(\omega_-,\omega_+)$ a measurable map $\f_{\xi_-,\xi_+}:\bd X\to X$ (defined $\nu_{\omega_+}$-everywhere). Hence we can pushforward the measure $\nu_{\omega_+}$ by $\f_{\xi_-,\xi_+}$ to get a map $B_-\times B_+\to \Prob(X)$. Now $\Prob(X)$ has a $\Gamma$-invariant metric (for example the Prokhorov metric). Hence by Double Isometric Ergodicity we get that $\Gamma$ fixes a probability measure on $X$. By countability of $X$, this implies that there is a finite set in $X$ which is $\Gamma$-invariant (the set of points with maximal measure), contradicting the assumption that the action is non-elementary.
\end{proof}

 \section{Regular Elements}

Regular elements are hyperbolic elements with strong contracting properties. In the irreducible case, they are exactly contracting isometries in the sense of \cite{BestvinaFujiwara}, and their existence is the main theorem of \cite{CapraceSageev}. For products, Caprace and Sageev \cite{CapraceSageev} show that such  elements provided the group is a lattice. In this section, we prove that such elements always exist for non-elementary actions, and moreover have some genericity property.

 We first recall the definition of contracting and regular isometries. 

\begin{definition}

\begin{itemize}
\item A geodesic line $\ell$ is called \emph{contracting} if there is $C>0$ such that any ball $B$ disjoint from $\ell$ projects to $\ell$ to a set of diameter less than $C$. 
\item If $X$ is irreducible, an element $g\in\Aut(X)$ is said \emph{contracting} if it is hyperbolic and one of its axis is contracting.
\item If $X$ is a product, an element $g\in\Aut(X)$ is said \emph{regular}  if it preserves each factor and if it acts as a contracting element on each irreducible factor.
\end{itemize}
\end{definition}

Our main tool in order to find regular elements is the following lemma of Caprace and Sageev \cite[Lemma 6.2]{CapraceSageev}:

\begin{lemma}\label{skewer -> contracting}
Assume that $g\in \Aut(X)$ is such that  $g.h \subsetneq h'$ for some  pair of strongly separated half-spaces $h\subset h'$. Then $g$ is a contracting isometry.
\end{lemma}


\begin{lemma}\label{skewerorflip}
Assume that $X$ is irreducible with a non-elementary and essential $\G$-action. Let $Z_n$ be a generic sequence for the random walk, $\xi\in \partial X$ be the limit of $Z_no$, and $s$ be a half-space containing $\xi$.

Then there exists an $N$ and $ s_2\subset s_1 \subset s$ pairwise strongly separated such that for every $n>N$ 
\begin{itemize}
\item either $Z_n s \subset s_2$,
\item or  $Z_n s\supset s_2^*$.
\end{itemize}
\end{lemma}

\begin{proof}

For notational simplicity, let $k_n=Z_n s$. Fix  $x\in s$ that is adjacent to the wall of $s$. Let $(s_m)_{m\geq 0}$ be an infinite descending chain of strongly separated half-spaces containing $\xi$, with $s=s_0$.  For each $m$ and $n$ large enough we have $Z_n x\in s_m$. Hence $k_n\cap s_m\neq\varnothing$. Furthermore, since $x$ is adjacent to the wall of $s$, we see that for $n$ large enough $\hat k_n\cap  s_m\neq \varnothing$. By strong separation, if we fix $m\geq 1$, we have that $\hat k_n \subset s_m$ for any $n$ large enough.

 This means that we either have $k_n\subset s_m$ or $k_n\supset s_m^*$. If $m\geq 2$, this gives the conclusion of the lemma.

\end{proof}

In the first case, we say that $Z_n$ is \emph{$s$-skewering}, in the second case that it is \emph{$s$-flipping}. By Lemma \ref{skewer -> contracting}, if $Z_n$ is $s$-skewering, then it is a contracting isometry.

In the following lemma, we use the stationary measure $\ch \nu$ for the random walk $\ch \mu$.

\begin{lemma}\label{lem:prop regular}
Let $s$ be a half-space. Then on a full measure set of $\omega\in \Omega$ we have that $Z_n(\omega) o $ converges to $\eta(\omega)$ and if $s\in U_{\eta(\omega)}$ then

 \begin{equation}\label{eq-proportion regular}
 \liminf_n\frac{1}{n} \vert \{k \leq n\mid Z_k \textrm{ is } s\textrm{-skewering } \}\vert\geq \ch\nu(s^*).
 \end{equation}

\end{lemma}

\begin{proof}
 We begin by observing that if $Z_n=g_1g_2\dots g_n$, then $Z_n^{-1}=g_n^{-1}\dots g_1^{-1}$, where $g_i$ follows the law $\mu$ and all of them are independent. 
 
 Recall that by Corollary \ref{stationary}, there is a unique stationary measure on $\~X$. This allows us to apply Corollary 2.7 of \cite{BenoistQuintBook} and therefore, for every continuous function $\f$ on $\overline X$, we have almost surely
$$\lim_{n\to+\infty}\frac 1n \sum_{k=1}^n \f(Z_k^{-1}x)=\ch\nu(\f).$$

 Let us fix a half-space $s$ and define $\f$ as the characteristic function of $s^*$. 
Observe that it is continuous  on $\overline X$. Now, fix $Z_n$ a generic sequence for $\f$.
By Lemma \ref{skewerorflip}, there is an $N$ such that for every $n>N$ either $Z_n$ is $s$-skewering or $s$-flipping. 
Fix $k$ with $N<k\leq n$. If $Z_k$ is $s$-flipping, then for $x\in s^*$, we have $x\in Z_k s$ (because $s^*\subset Z_k s$), hence $Z_k^{-1}x\in s$. It follows that $\f(Z_k^{-1}x)=0$. Therefore
$$ \sum_{k=1}^n \f(Z_k^{-1}x)\leq \vert \{k \leq n\mid Z_k \textrm{ is } s\textrm{-skewering } \}\vert+N,$$
and Equation \eqref{eq-proportion regular} follows.

Finally, the fact Equation \eqref{eq-proportion regular} holds almost surely follows by recalling that there are countably many half-spaces, and each of the corresponding sets has full measure.

\end{proof}

\begin{theorem}\label{reg elements irred}
Let $X$ be irreducible, and $\Gamma$ act on $X$ essentially and non-elementarily. Then almost surely 
$$\lim_{n\to +\infty}\frac{1}{n} \vert \{k \leq n\mid Z_k \textrm{ is contracting }  \}\vert=1.$$
\end{theorem}

\begin{proof}

Fix a generic sequence $(Z_n)$ as provided by Lemma \ref{lem:prop regular}, with limit $Z_n o \to\eta$. Then since $X$ is irreducible and $\eta$ is regular, we have that $\{\eta\}=\bigcap_{m\geq 0} s_m$, for some descending chain $(s_m)$. It follows that $\~X\setminus \{\eta\}=\bigcup_{m\geq 0} s_m^*$, and since $(s_m^*)$ is ascending, we have that $$\sup_m\ch\nu(s_m^*)=1.$$

Now, Lemma \ref{lem:prop regular} assures us that for every $m$
$$\liminf_n   \frac{1}{n} \vert \{k \leq n\mid Z_k \textrm{ is contracting } \}\vert\geq \ch\nu(s_m^*)$$
for every $m$. Since this proportion is at most 1, the sequence is in fact convergent and we get the result.
\end{proof}

\begin{remark}
The proof above gives slightly more: namely, for every half-space $s$, there is a positive measure set of $(Z_n)$ such that  $Z_n$ is $s$-skewering with frequency at least $\ch\nu(s)$. Indeed, the probability that this occurs is at least $\nu(s)$.
\end{remark}

\begin{theorem}\label{reg elements!!}
Assume that the action of $\Gamma$ is non-elementary, essential and stabilizes each irreducible factor of $X$. Then almost surely

$$\lim_{n\to +\infty}\frac{1}{n} \vert \{k \leq n\mid Z_k \textrm{ is regular }  \}\vert=1.$$

\end{theorem}

\begin{proof}
Let $X=X_1\times\dots\times X_d$ be the decomposition of $X$ into irreducible factors. Applying Theorem \ref{reg elements irred} to the action of $\Gamma$ on each factor, we find $N$ such that for $n>N$ we have for every factor $X_i$ of $X$, the set of $k\leq n$ such that $Z_k$ is contracting on $X_i$ is of cardinality at least $n(1-\eps)$. It follows that there are at least $n(1-d\eps)$ elements $Z_k$ which are contracting simultaneously on each factor.  
\end{proof}

In terms of the probability that a given element is regular, we deduce the following:

\begin{corollary}\label{RegularEverywhere!}
Under the same assumptions as Theorem \ref{reg elements!!}, we have
$$\frac{1}{n}\sum_{k=1}^n \Ps(Z_k \textrm{ is regular})=1.$$
\end{corollary}

\begin{proof}
Let $f:\Gamma\to \{0,1\}$ be the characteristic function of the set of regular elements. By Theorem \ref{reg elements!!}, we have $\frac{1}{n} \sum\limits_{k=1}^n f(Z_k)\to 1$ almost surely.

Taking the expectation, we get
$$\lim_{n\to +\infty}\frac{1}{n} \sum_{k=1}^n \mathbb E(f(Z_k))= 1,$$
which is the desired result since $\mathbb E(f(Z_k))=\Ps(Z_k \textrm{ is regular})$.
\end{proof}

When the group does not stabilize each factor, the limit might be smaller, due to the fact that there is a positive proportion of elements which do not stabilize each factor, hence cannot be regular. However, we can say the following.

\begin{corollary}\label{reg elements permuting}

We have almost surely

$$\liminf_{n\to +\infty}\frac{1}{n} \vert \{k \leq n\mid Z_k \textrm{ is regular }  \}\vert>0.$$
\end{corollary}
\begin{proof}
Let $\Gamma_0\lhd \Gamma$ be the finite index normal subgroup which stabilizes each factor. Let $\f(n)$ be the subsequences formed by indices such that $Z_{\f(n)}$ belong to $\Gamma_0$. Then $Z_{\f(n)}$ is a random walk on $\Gamma_0$ of law $\mu_0$ (the first return probability), so that Theorem \ref{reg elements!!} apply and proves that almost surely 

$$\lim_{n\to +\infty}\frac{1}{\f(n)} \vert \{k \leq n\mid Z_{\f(k)} \textrm{ is regular }  \}\vert=1.$$

By Lemma \ref{phi(n)/n}, we know that $\f(n)/n$ almost surely has a positive limit. The result follows.

\end{proof}

\bibliographystyle{alpha}
\bibliography{biblio}

\newcommand{\etalchar}[1]{$^{#1}$}
\begin{thebibliography}{BCG{\etalchar{+}}09}

\bibitem[Ago13]{Agol}
Ian Agol.
\newblock The virtual {H}aken conjecture.
\newblock {\em Doc. Math.}, 18:1045--1087, 2013.
\newblock With an appendix by Agol, Daniel Groves, and Jason Manning.

\bibitem[Atk76]{Atk76}
Giles Atkinson.
\newblock Recurrence of co-cycles and random walks.
\newblock {\em J. London Math. Soc. (2)}, 13(3):486--488, 1976.

\bibitem[Bal89]{Ballmann89}
Werner Ballmann.
\newblock On the {D}irichlet problem at infinity for manifolds of nonpositive
  curvature.
\newblock {\em Forum Math.}, 1(2):201--213, 1989.

\bibitem[BC12]{Behrstock_Charney}
Jason Behrstock and Ruth Charney.
\newblock Divergence and quasimorphisms of right-angled {A}rtin groups.
\newblock {\em Math. Ann.}, 352(2):339--356, 2012.

\bibitem[BCG{\etalchar{+}}09]{BCGNW}
J.~Brodzki, S.~J. Campbell, E.~Guentner, G.~A. Niblo, and N.~J. Wright.
\newblock Property {A} and {$\rm CAT(0)$} cube complexes.
\newblock {\em J. Funct. Anal.}, 256(5):1408--1431, 2009.

\bibitem[BF09]{BestvinaFujiwara}
Mladen Bestvina and Koji Fujiwara.
\newblock A characterization of higher rank symmetric spaces via bounded
  cohomology.
\newblock {\em Geom. Funct. Anal.}, 19(1):11--40, 2009.

\bibitem[BF14]{BFICM}
Uri Bader and Alex Furman.
\newblock Boundaries, rigidity of representations, and lyapunov exponents.
\newblock {\em Proceedings of ICM 2014}, pages 71--96, 2014.

\bibitem[BH99]{Bridson_Haefliger}
Martin~R. Bridson and Andr{\'e} Haefliger.
\newblock {\em Metric spaces of non-positive curvature}, volume 319 of {\em
  Grundlehren der Mathematischen Wissenschaften [Fundamental Principles of
  Mathematical Sciences]}.
\newblock Springer-Verlag, Berlin, 1999.

\bibitem[BQ]{BenoistQuintBook}
Y.~Benoist and J.-F. Quint.
\newblock {\em Random Walks on Reductive Groups}.

\bibitem[BQ11]{BenoistQuintStationnaire}
Yves Benoist and Jean-Fran{\c{c}}ois Quint.
\newblock Mesures stationnaires et ferm\'es invariants des espaces homog\`enes.
\newblock {\em Ann. of Math. (2)}, 174(2):1111--1162, 2011.

\bibitem[BQ16]{BenoistQuintHyp}
Y.~Benoist and J.-F. Quint.
\newblock Central limit theorem on hyperbolic groups.
\newblock {\em Izv. Ross. Akad. Nauk Ser. Mat.}, 80(1):5--26, 2016.

\bibitem[BS06]{BaderShalom}
Uri Bader and Yehuda Shalom.
\newblock Factor and normal subgroup theorems for lattices in products of
  groups.
\newblock {\em Invent. Math.}, 163(2):415--454, 2006.

\bibitem[CFI12]{CFI}
I.~{Chatterji}, T.~{Fern{\'o}s}, and A.~{Iozzi}.
\newblock {The Median Class and Superrigidity of Actions on CAT(0) Cube
  Complexes}.
\newblock {\em ArXiv e-prints}, December 2012.

\bibitem[CL10]{CapraceLytchak}
Pierre-Emmanuel Caprace and Alexander Lytchak.
\newblock At infinity of finite-dimensional {CAT}(0) spaces.
\newblock {\em Math. Ann.}, 346(1):1--21, 2010.

\bibitem[CN05]{Chatterji_Niblo}
Indira Chatterji and Graham Niblo.
\newblock From wall spaces to {$\rm CAT(0)$} cube complexes.
\newblock {\em Internat. J. Algebra Comput.}, 15(5-6):875--885, 2005.

\bibitem[CS11]{CapraceSageev}
Pierre-Emmanuel Caprace and Michah Sageev.
\newblock Rank rigidity for {CAT}(0) cube complexes.
\newblock {\em Geom. Funct. Anal.}, 21(4):851--891, 2011.

\bibitem[CZ13]{CapraceZadnik}
Pierre-Emmanuel Caprace and Ga{\v{s}}per Zadnik.
\newblock Regular elements in {${\rm CAT}(0)$} groups.
\newblock {\em Groups Geom. Dyn.}, 7(3):535--541, 2013.

\bibitem[Far03]{Farley}
Daniel~S. Farley.
\newblock Finiteness and {$\rm CAT(0)$} properties of diagram groups.
\newblock {\em Topology}, 42(5):1065--1082, 2003.

\bibitem[{Fer}15]{Fernos}
T.~{Fern{\'o}s}.
\newblock {The Furstenberg Poisson Boundary and CAT(0) Cube Complexes}.
\newblock {\em accepted to Ergodic Theory and Dynamical Systems, available on
  ArXiv e-prints}, July 2015.

\bibitem[Fur63]{Furstenberg1963}
Harry Furstenberg.
\newblock Noncommuting random products.
\newblock {\em Trans. Amer. Math. Soc.}, 108:377--428, 1963.

\bibitem[Fur71]{Furstenberg}
Harry Furstenberg.
\newblock Boundaries of {L}ie groups and discrete subgroups.
\newblock In {\em Actes du {C}ongr\`es {I}nternational des {M}ath\'ematiciens
  ({N}ice, 1970), {T}ome 2}, pages 301--306. Gauthier-Villars, Paris, 1971.

\bibitem[Fur02]{Furman}
Alex Furman.
\newblock Random walks on groups and random transformations.
\newblock In {\em Handbook of dynamical systems, {V}ol.\ 1{A}}, pages
  931--1014. North-Holland, Amsterdam, 2002.

\bibitem[GK15]{KarlssonGouezel}
S.~{Gou{\"e}zel} and A.~{Karlsson}.
\newblock {Subadditive and Multiplicative Ergodic Theorems}.
\newblock {\em ArXiv e-prints}, September 2015.

\bibitem[GM12]{GauteroMatheus}
Fran{\c{c}}ois Gautero and Fr{\'e}d{\'e}ric Math{\'e}us.
\newblock Poisson boundary of groups acting on {$\Bbb R$}-trees.
\newblock {\em Israel J. Math.}, 191(2):585--646, 2012.

\bibitem[GR85]{GR85}
Y.~Guivarc'h and A.~Raugi.
\newblock Fronti\`ere de {F}urstenberg, propri\'et\'es de contraction et
  th\'eor\`emes de convergence.
\newblock {\em Z. Wahrsch. Verw. Gebiete}, 69(2):187--242, 1985.

\bibitem[Gui80]{Guivarc'h1980}
Y.~Guivarc'h.
\newblock Sur la loi des grands nombres et le rayon spectral d'une marche
  al\'eatoire.
\newblock In {\em Conference on {R}andom {W}alks ({K}leebach, 1979)
  ({F}rench)}, volume~74 of {\em Ast\'erisque}, pages 47--98, 3. Soc. Math.
  France, Paris, 1980.

\bibitem[Gur07]{Guralnik}
Dan~P. Guralnik.
\newblock Coarse decompositions for boundaries of cat(0) groups.
\newblock 2007.
\newblock Preprint.

\bibitem[Ka{\u\i}87]{KaimanovichOseledts}
V.~A. Ka{\u\i}manovich.
\newblock Lyapunov exponents, symmetric spaces and a multiplicative ergodic
  theorem for semisimple {L}ie groups.
\newblock {\em Zap. Nauchn. Sem. Leningrad. Otdel. Mat. Inst. Steklov. (LOMI)},
  164(Differentsialnaya Geom. Gruppy Li i Mekh. IX):29--46, 196--197, 1987.

\bibitem[Kai91]{Kaima91}
Vadim~A. Kaimanovich.
\newblock Poisson boundaries of random walks on discrete solvable groups.
\newblock In {\em Probability measures on groups, {X} ({O}berwolfach, 1990)},
  pages 205--238. Plenum, New York, 1991.

\bibitem[Kai00]{Kaimanovich}
Vadim~A. Kaimanovich.
\newblock The {P}oisson formula for groups with hyperbolic properties.
\newblock {\em Ann. of Math. (2)}, 152(3):659--692, 2000.

\bibitem[Kai03]{KaimanovichErgodic}
V.~A. Kaimanovich.
\newblock Double ergodicity of the {P}oisson boundary and applications to
  bounded cohomology.
\newblock {\em Geom. Funct. Anal.}, 13(4):852--861, 2003.

\bibitem[KM96]{KaimanovichMasur}
Vadim~A. Kaimanovich and Howard Masur.
\newblock The {P}oisson boundary of the mapping class group.
\newblock {\em Invent. Math.}, 125(2):221--264, 1996.

\bibitem[KM99]{KarlssonMargulis}
Anders Karlsson and Gregory~A. Margulis.
\newblock A multiplicative ergodic theorem and nonpositively curved spaces.
\newblock {\em Comm. Math. Phys.}, 208(1):107--123, 1999.

\bibitem[KS15]{KarSageev}
A.~{Kar} and M.~{Sageev}.
\newblock {Ping pong on CAT(0) cube complexes}.
\newblock {\em ArXiv e-prints}, July 2015.

\bibitem[Led01]{Led01}
Fran{\c{c}}ois Ledrappier.
\newblock Some asymptotic properties of random walks on free groups.
\newblock In {\em Topics in probability and {L}ie groups: boundary theory},
  volume~28 of {\em CRM Proc. Lecture Notes}, pages 117--152. Amer. Math. Soc.,
  Providence, RI, 2001.

\bibitem[Lin10]{Link}
Gabriele Link.
\newblock Asymptotic geometry in products of {H}adamard spaces with rank one
  isometries.
\newblock {\em Geom. Topol.}, 14(2):1063--1094, 2010.

\bibitem[{Mar}15]{Martin}
A.~{Martin}.
\newblock {On the cubical geometry of Higman's group}.
\newblock {\em ArXiv e-prints}, June 2015.

\bibitem[MT14]{MaherTiozzo}
J.~{Maher} and G.~{Tiozzo}.
\newblock {Random walks on weakly hyperbolic groups}.
\newblock {\em ArXiv e-prints}, October 2014.

\bibitem[Nic04]{Nica}
Bogdan Nica.
\newblock Cubulating spaces with walls.
\newblock {\em Algebr. Geom. Topol.}, 4:297--309, 2004.

\bibitem[NS13]{NevoSageev}
Amos Nevo and Michah Sageev.
\newblock The {P}oisson boundary of {${\rm CAT}(0)$} cube complex groups.
\newblock {\em Groups Geom. Dyn.}, 7(3):653--695, 2013.

\bibitem[Ose68]{Oseledec}
V.~I. Oseledec.
\newblock A multiplicative ergodic theorem. {C}haracteristic {L}japunov,
  exponents of dynamical systems.
\newblock {\em Trudy Moskov. Mat. Ob\v s\v c.}, 19:179--210, 1968.

\bibitem[Wis09]{WiseMalnormal}
Daniel~T. Wise.
\newblock Research announcement: the structure of groups with a quasiconvex
  hierarchy.
\newblock {\em Electron. Res. Announc. Math. Sci.}, 16:44--55, 2009.

\end{thebibliography}

\end{document}